\documentclass[12pt]{article}

\usepackage{graphicx}
\usepackage{latexsym,amssymb}
\usepackage{amsthm}
\usepackage{indentfirst}
\usepackage{amsmath}
\usepackage{color}
\usepackage{xcolor}
\usepackage{fourier}
\usepackage[all]{xy}

\newtheorem{theoremalph}{Theorem}

\newtheorem*{Main Theorem}{Main Theorem}

\newtheorem{Theorem}{Theorem}[section]
\newtheorem*{Theorem A}{Theorem A}
\newtheorem*{Theorem A'}{Theorem A'}
\newtheorem*{Theorem B'}{Theorem B'}

\newtheorem{Definition}[Theorem]{Definition}
\newtheorem{Proposition}[Theorem]{Proposition}
\newtheorem{Lemma}[Theorem]{Lemma}

\newtheorem{Remark}[Theorem]{Remark}

\newtheorem{Remark-numbered}[Theorem]{Remark}
\newtheorem{Corollary}[Theorem]{Corollary}

\newtheorem*{Claim}{Claim}
\newtheorem{Claim-numbered}[Theorem]{Claim}

 \def\NN{{\mathbb N}} 

 \def\RR{{\mathbb R}}

\def\La{\Lambda}

  \def\cG{{\cal G}} \def\cM{{\cal M}} 
  \def\cH{{\cal H}} \def\cN{{\cal N}} 
\def\cC{{\cal C}}  \def\cI{{\cal I}}  \def\cU{{\cal U}}
   \def\cP{{\cal P}} \def\cV{{\cal V}}
\def\cE{{\cal E}}    
\def\cF{{\cal F}}   \def\cR{{\cal R}} \def\cX{{\cal X}}

\newcommand{\sing}{{\operatorname{Sing}}}

\newcommand{\orb}{\operatorname{Orb}}

\def\dim{\operatorname{dim}}
\def\ind{\operatorname{Ind}}

\def\Sing{\operatorname{Sing}}
\def\orb{\operatorname{Orb}}
\def\supp{\operatorname{Supp}}
\def\ud{\operatorname{d}}
\def\e{{\varepsilon}}

\def\Det{\operatorname{det}}

\def\Leb{\operatorname{Leb}}

\def\Basin{\operatorname{Basin}}

\def\wh{\widehat}

\begin{document}

\title{On physical measures of multi-singular hyperbolic vector fields}

\author{Sylvain Crovisier, Xiaodong Wang, Dawei Yang and  Jinhua Zhang\footnote{S. Crovisier   was partially supported by  the ERC project 692925 \emph{NUHGD}.
X. Wang was partially supported by National Key R\&D Program of China (2021YFA1001900), NSFC 12071285 and Innovation Program of Shanghai Municipal Education Commission (No. 2021-01-07-00-02-E00087).
D. Yang  was partially supported by National Key R\&D Program of China (2022YFA1005801), NSFC 12171348 and NSFC 12325106.
J. Zhang was partially supported by National Key R\&D Program of China (2021YFA1001900), National Key R\&D Program of China (2022YFA1005801), NSFC 12001027 and the Fundamental Research Funds for the Central Universities.}}

\maketitle

\begin{abstract} 
Bonatti and da Luz have introduced the class of \emph{multi-singular hyperbolic} vector fields to characterize
systems whose periodic orbits and singularities do not bifurcate under perturbation
(called star vector fields).

In this paper, we study the Sina\"{\i}-Ruelle-Bowen measures for multi-singular hyperbolic vector fields: in  a
$C^1$ open and $C^1$ dense subset of multi-singular hyperbolic vector fields, each {$C^\infty$} one admits  \emph{finitely} many physical measures whose basins cover a \emph{full} Lebesgue measure subset of the manifold. Similar results are also obtained for $C^1$ generic multi-singular hyperbolic vector fields.
\medskip

%
\end{abstract}
\tableofcontents

{
\section{Introduction}

One aim in differentiable dynamical systems is to describe the behavior of most orbits of typical systems. From the ergodic viewpoint, this has been realized by Sina\"{\i}-Ruelle-Bowen \cite{Sinai,Bowen-book,BR,Ruelle} for uniformly hyperbolic systems. In the case of hyperbolic vector fields, it has been proved by Bowen and Ruelle \cite{BR} that the forward orbit of Lebesgue almost every point in the manifold will converge to an ergodic measure
which belongs to a finite family of (physical) measures and which admits rich geometric and dynamical properties. Such measures are nowadays called Sina\"{\i}-Ruelle-Bowen (SRB) measures.

It is expected that such theory can be extended for large classes of systems beyond uniform hyperbolicity.
For instance this has been established for some sets of non-hyperbolic parameters with positive Lebesgue measure
inside the H\'enon family. This also has led to many important progresses in the study of partially hyperbolic systems. However there is still no complete theory.

In the case of vector fields, Morales, Pacifico and Pujals \cite{MPP} have defined the notion of ``singular hyperbolicity'' to characterize robust transitivity in dimension $3$ and singular dynamics modeled on the Lorenz attractor. Singular hyperbolicity is still a form of partial hyperbolicity allowing the existence of singularities,   i.e.,  points where the vector field vanishes. More recently, Bonatti and da Luz \cite{BdL} have introduced the ``multi-singular hyperbolicity'' to describe a robust class of vector fields satisfying a strong stability property (the star property): bifurcation phenomena cannot occur on   periodic orbits nor on singularities for any system $C^1$ close. The star property has played an important role in the exploration of the structural stability conjecture of Palis and Smale.
It has been shown that the star property coincides with the uniform hyperbolicity in the case of diffeomorphisms
and non-singular vector fields~\cite{L80,Ma82,Ma88,Ha,GW}.

The dynamics of general star vector field is still striking due to the existence of singularities.
There are many examples of such systems which are not uniformly hyperbolic:
for instance, some Cherry flows \cite{C, Palis-de Melo} in dimension $2$ and the famous Lorenz attractors \cite{Lo}
in dimension $3$ (which supports robust chaotic non-hyperbolic dynamics).
In dimension 4, Shi, Gan and Wen~\cite{SGW} have proved that $C^1$ generic star vector fields are singular hyperbolic. 
In dimension $5$, da Luz~\cite{dL} has constructed a $C^1$ open set of star vector fields which are not singular hyperbolic. From \cite{BdL} it turns out that multi-singular hyperbolicity characterizes star vector fields for a $C^1$ open and dense set of systems. The multi-singular hyperbolicity is specific to singular vector fields (i.e. has no counterpart for diffeomorphisms) and is not covered by partial hyperbolicity.

In this paper, we study the ergodic theory of multi-singular hyperbolic vector fields: we will prove that for a large set among star vector fields, there are finitely many physical SRB measures whose basins cover a set with full Lebesgue measure in the manifold.
\medskip

Let $M$ be a $C^\infty$ closed Riemannian manifold. Denote by $\cX^r(M)$ ($r\geq 1$) the space of $C^r$ vector fields on $M$. 
As mentioned previously, a $C^1$ vector field $X$ is said to be \emph{star} if there exists a $C^1$ neighborhood $\cU$ of $X$ such that, for any $Y\in\cU$, all the critical elements (i.e. periodic orbits or singularities) of the flow generated by $Y$ are hyperbolic. We denote by $\cX^*(M)$ the set of multi-singular hyperbolic vector fields on $M$ endowed with $C^1$ topology.
On a $C^1$ open and dense subset of $\cX^1(M)$, the star property is equivalent~\cite{BdL} to the multi-singular hyperbolicity, whose precise definition is stated in Section~\ref{ss.definition-multi} below. In this paper we study the dynamics for a $C^1$ open and dense set of vector fields in $\cX^1(M)$, hence we do not have to distinguish star and multi-singular hyperbolic systems.

Let $\varphi=(\varphi_t)_{t\in\RR}$ denote the flow generated by a vector field $X$. 
A $\varphi$-invariant probability measure is said to be \emph{physical} if it attracts a set with positive Lebesgue measure, i.e.,
${\rm Leb}\big(\Basin(\mu)\big)>0$, where
$$\Basin(\mu)=\bigg\{x\in M: \lim_{t\rightarrow+\infty}\frac{1}{t}\int_{0}^t\delta_{\varphi_s(x)}\ud s=\mu \bigg\}.$$

This is related to ergodic measures satisfying variational and geometric properties.
An ergodic measure $\mu$ is said to be a \emph{Sina\"{\i}-Ruelle-Bowen measure} (abbr. SRB measure) if it has at least one positive Lyapunov exponent and if its metric entropy equals the sum of its positive Lyapunov exponents. Note that  SRB measures are defined  in $C^1$-scenario, and this allows us to work on SRB measures in $C^1$-generic setting.
A classical theorem by Ledrappier and Young \cite{LY} states that for $C^{2}$ systems, the SRB property is equivalent to the absolute continuity of the conditional measures along the unstable manifolds which is usually taken as the definition of SRB measures for $C^{2}$ systems.
For any $C^2$ vector field, an ergodic hyperbolic SRB measure is physical~\cite[Th\'eor\`eme 4.9]{Le}. 

\begin{theoremalph}\label{Thm:main-ergodic}
There exists a $C^1$ open and dense subset $\cU$ of $\cX^*(M)$ such that for any $C^\infty$ vector field $Y\in \cU$, there exist 
finitely many ergodic  physical measures $\mu_1,\cdots,\mu_k$ such that 
	$$\Leb\big(\bigcup_{i=1}^k\Basin(\mu_i)\big)=\Leb(M).$$
\end{theoremalph}

This property extends to $C^1$ generic multi-singular hyperbolic vector fields.
\begin{theoremalph}~\label{Thm:main-generic-ergodic}
There exists a dense G$_\delta$ subset $\cG$ of $\cX^*(M)$ such that each $X\in\cG$ has finitely many ergodic physical measures $\mu_1,\cdots,\mu_k$ such that $$\Leb\big(\bigcup_{i=1}^k\Basin(\mu_i)\big)=\Leb(M).$$
\end{theoremalph}

Beyond uniform hyperbolicity, the $C^1$ generic ergodic theory was established for partially hyperbolic volume-preserving diffeomorphisms by Avila-Bochi-Wilkinson \cite{ABW}, Avila-Bochi \cite{AB} and Avila-Crovisier-Wilkinson \cite{ACW}. 
For dissipative cases,  there are several results concerning the existence adnd finiteness   of
physical measures for  partially hyperbolic diffeomorphisms with (1) mostly contracting/expanding center or (2) with one-dimensional center, among which we would like to mention that  the results in 
~\cite{CYZ, HYY} are obtained via  variational approach as in our paper.
Using SRB measures, our result characterizes a large class of dissipative  vector fields including the famous   Lorenz attractors and some others.

In Theorems~\ref{Thm:main-ergodic} and~\ref{Thm:main-generic-ergodic}
we have more information.
For any vector field $X$ in the open set $\cU$ provided by Theorem~\ref{Thm:main-ergodic},
there exist  finitely many attractors $\Lambda_1,\dots,\Lambda_k$, which are sinks or
singular hyperbolic attractors.
If moreover $X$ is $C^\infty$ or belongs to the dense G$_\delta$ subset $\cG$
given by Theorem~\ref{Thm:main-generic-ergodic},
the support of each ergodic physical measure $\mu_i$ coincides with one of these attractors $\Lambda_i$
and each attractor supports exactly one physical measure.
Moreover when $\Lambda_i$ is not a sink, $\mu_i$ is an SRB measure.
See the more detailed statement in Theorem~\ref{Thm:main-restated} below.

Let us notice that in the particular case of $C^2$ singular hyperbolic vector fields,
any system admits a finite collection of physical measures whose basins cover a set with full Lebesgue measure,
see~\cite{APPV,LeYa,CYZ,A}.

\smallskip

Let us discuss the main ingredients of our proofs.
There are several ways to build SRB measures for differentiable systems.
We follow a variational approach, which goes back to~\cite{CQ,Q}
and which is well-adapted to systems admitting a dominated splitting, see~\cite{CCE,CaY,LeYa,CYZ,HYY};
when there is no dominated splitting, results have been obtained for some smooth systems
using Yomdin theory~\cite{Bu,B2,BCS}. Singular hyperbolic vector fields admit a dominated splitting,
which allowed~\cite{LeYa,CYZ} to build SRB measures in this case. For multi-singular hyperbolic vector fields,
there is no dominated splitting in general and we had to face a new difficulty.
In~\cite{Bu}, the variational approach combined with Yomdin theory (for $C^\infty$ vector fields) produces measures whose entropy
is linked with the (upper) Lyapunov exponents of points in a set with positive Lebesgue measure:
without dominated splitting, this is not enough to guarantee that these measures are SRB.
Under multi-singular hyperbolicity, one can blow-up each singularity and obtain a dominated splitting over a bundle
related to the initial tangent bundle, which allows us to check Pesin entropy formula.
This gives the following result for multi-singular hyperbolic sets.
For an invariant compact set $\Lambda$, its basin is defined as:
$$\Basin(\Lambda)=\big\{x\in M: \omega(x)\subset \Lambda \big\}.$$

\begin{theoremalph}~\label{Thm:main-localized}
Let $X\in\cX^\infty(M)$ and $\Lambda$ be a multi-singular hyperbolic set which does not contain any sink. 
If $\Leb\big(\Basin(\Lambda)\big)>0$, then there is an SRB measure supported on $\Lambda$.
\end{theoremalph}

This is the starting point for proving Theorems~\ref{Thm:main-ergodic} and~\ref{Thm:main-generic-ergodic}.
Their proofs use the robustness of singular hyperbolic attractors (proved in~\cite{CY})
and a characterization of attractors in the multi-singular hyperbolic setting
using bi-Pliss points (see Theorem~\ref{p.unstable-manifold-attractor}).}

\subsection*{Acknowledgments}
We would like to thank the anonymous referee for her/his valuable comments which improves a lot of the paper.
We would like to thank Yi Shi for many useful discussions and comments.

\section{Fundamental  properties of  multi-singular hyperbolicity}
Let $M$ be a $d$-dimensional closed manifold endowed with a Riemannian metric.
The set of transverse intersections between two submanifolds $V,W$ will be denoted by $V\pitchfork W$.
Given a $C^1$ vector field $X$ on $M$, we denote by $(\varphi^X_t)_{t\in\RR}$  the flow generated by $X$ and $(D\varphi_t^X)_{t\in\RR}$  the tangent flow. When there is no ambiguity, we will drop the index $X$ for simplicity.
The \emph{singular set} is $\sing(X):=\{x: X(x)=0\}$.
A point $x$ is \emph{periodic} if it is non singular and satisfies $\varphi_t(x)=x$ for some $t>0$.
A point $x$ is a \emph{critical point} if it is either a singularity or a periodic point.
The orbit of a critical point is a \emph{critical element}.

For each point $x\in M\setminus\sing(X)$, one can define the normal space at $x$,
$$\cN^X(x)=\cN(x):=\{v\in T_xM: v\bot X(x)\}.$$
This gives a bundle $\cN$ over the set $M\setminus\sing(X)$ that  is called  the \emph{normal bundle}.
\begin{Definition}
For a $C^1$ vector field $X$, the ~\emph{linear Poincar\'e flow} $(\Psi_t)_{t\in\RR}$ is defined on $\cN$ by
$$\Psi_t(v)=D\varphi_t(v)-\frac{<D\varphi_t(v), X(\varphi_t(x))>}{\|X(\varphi_t(x))\|^2}\cdot X(\varphi_t(x)) \textrm{~~ for any $x\in M\setminus\sing(X)$ and $v\in\cN(x)$}.$$ 	
\end{Definition}
In other terms, the linear Poincar\'e flow is defined as follow: for any vector in the normal bundle, one first sends it by the tangent flow and then projects it back to the normal bundle.

\subsection{Chain recurrence classes}\label{ss.chain-recurrence}
Conley's theory~\cite{Conley} decomposes the dynamics as follows.
One says that $x$ is \emph{chain attainable} from $y$ (and one denotes $y\dashv x$),
if there exists $T>0$ such that for any $\varepsilon>0$
there are $\{x_i\}_{i=0}^k$ with $k\geq 1$ and $t_i>T$ for $0\leq i<k$ satisfying:
\begin{itemize}
	\item $x_0=y$ and $x_k=x$;
	\item $\ud(\varphi_{t_i}(x_i),x_{i+1})<\e$ for $i=0,\cdots,k-1.$
\end{itemize}
A point $x$ is ~\emph{chain recurrent} if $x$ is chain attainable from itself. We denote by $\cR(X)$ the set of chain recurrent points. One defines a relation $\sim$ on $\cR(X)$ as follows
$$x\sim y \;\Leftrightarrow\; (y\dashv x~\text{ and }~x\dashv y). $$
This is a closed equivalence relation. For each point $x\in\cR(X)$, the equivalence class containing $x$ is called the \emph{chain recurrence class} of $x$ and will be denoted by $C(x)$.

The basin of an invariant compact set $\Lambda$ is defined as 
$$\Basin(\Lambda)=\big\{x\in M: \omega(x)\subset\Lambda \big\}.$$

An invariant compact set $\Lambda$ is \emph{attracting}
if there exists a neighborhood $U$ of $\Lambda$ such that $\varphi_t(\overline U)\subset U$ for some $t>0$
and $\Lambda=\bigcap_{t\in\mathbb{R}}\varphi_t(U)$.
In this case, one can replace $U$ by a neighborhood which is \emph{trapping}
i.e. which satisfies $\varphi_t(\overline U)\subset{\rm Int}(U),\forall t>0$, see~\cite[Section 5.B]{Conley}.
An invariant compact set $\Lambda$ is an \emph{attractor} if
it is an attracting transitive set.

An invariant compact set $\Lambda$ is {\it Lyapunov stable} if
for any neighborhood $U$ of $\Lambda$, there exists a smaller neighborhood $V$ of $\Lambda$ such that
$\varphi_t(V)\subset U$ for any $t\geq 0$.

A \emph{quasi-attractor} is a chain recurrence class which admits arbitrarily small trapping neighborhoods.

By~\cite[Proposition 1.7]{BC}, for any vector field $X$ in a dense G$_\delta$ subset of $\cX^1(M)$, a chain recurrence class $C$ of $X$ is a quasi-attractor if and only if $C$ is Lyapunov stable; moreover it also gives
the following property (see also \cite[Theorem A]{MP}).

\begin{Theorem}\label{Theo:generic}
There exists a dense G$_\delta$ subset $\cG$ of $\cX^1(M)$ and for any $X\in\cG$, there exists a dense G$_\delta$ subset $\cR_X\subset M$ such that for any $x\in \cR_X$, the limit set $\omega(x)$ is a quasi-attractor.
\end{Theorem}

The following result states a local version of this property.
It follows easily from previous works and its proof will be given in the appendix.
\begin{Theorem}\label{Thm:go-to-Lyapunov}
	There exists a dense G$_\delta$ subset $\cG\subset\cX^1(M)$ with the following property.
	 For any $X\in\cG$ and any hyperbolic critical element $\gamma$ of $X$, 
  there is a dense G$_\delta$ subset $R_\gamma$ of $W^u(\gamma)$ such that the $\omega$-limit set of any point $x\in R_\gamma$ is   a quasi-attractor.
\end{Theorem}

\subsection{Multi-singular hyperbolicity}\label{ss.definition-multi}
We first recall the notions of hyperbolicity and singular hyperbolicity.
\begin{Definition}\label{Def:hyperbolicity}
Let $X\in\mathcal{X}^1(M)$.  A $(\varphi_t)_{t\in\mathbb{R}}$-invariant compact set $\La$ is \emph{hyperbolic} if there exist a $(D\varphi_t)$-invariant continuous splitting $T_\Lambda M=E^s\oplus \mathbb{R}\cdot X\oplus E^u$ and two constants $C>0,\lambda>0$ satisfying the following properties. 
\begin{itemize}
\item $E^s$ is $(D\varphi_t)_{t\in\mathbb{R}}$-\emph{contracted}: $$\|D\varphi_t|_{E^{s}(x)}\|\leq C\cdot e^{-\lambda t}~~ \text{for all $x\in \La$ and $t\geq0$}; $$
\item $E^u$ is $(D\varphi_t)_{t\in\mathbb{R}}$-\emph{expanded}: $$\|D\varphi_{-t}|_{E^{u}(x)}\|\leq C\cdot e^{-\lambda t}~~ \text{for all $x\in \La$ and $t\geq0$}.$$
\end{itemize}
\end{Definition}

\begin{Remark}\label{Rem:hyperbolicity}
\begin{enumerate}
\item The sub-bundles $E^s$, $\mathbb{R}\cdot X$ and $E^u$ are allowed to be trivial. For instance, when $\Lambda$ is reduced to   a singularity $\sigma$ (i.e. $ X(\sigma) =0$), then the splitting is $T_\sigma M=E^s_\sigma\oplus E^u_\sigma$. When $\Lambda$ is reduced to a critical element (a singularity or a periodic orbit)
such that  $E^s=\{0\}$ (resp. $E^u=\{0\}$), then $\Lambda$ is a source (resp. sink).
\item In the hyperbolic splitting $T_\Lambda M=E^s\oplus \RR\cdot X \oplus E^u$, the dimension $\dim(E^s)$ is called the \emph{index} of $\Lambda$. 
\end{enumerate}
\end{Remark}

In their study of robustly transitive singular sets, Morales,  Pacifico and  Pujals \cite{MPP} introduced the notion of \emph{singular hyperbolicity} which very well describes the structure of Lorenz attractors.

\begin{Definition}
Let $X\in\mathcal{X}^1(M)$.  A $(\varphi_t)_{t\in\mathbb{R}}$-invariant compact set $\La$ is \emph{positively singular hyperbolic} if there exist a $(D\varphi_t)_{t\in\mathbb{R}}$-invariant splitting $T_\La M=E^{ss}\oplus E^{cu}$,   and two numbers $C>1$ and $\lambda>0$ satisfying the following properties.
\begin{itemize} 
\item Contraction of $E^{ss}$: $\|D\varphi_t|_{E^{ss}(x)}\|\leq C\cdot e^{-\lambda t}$ for all $x\in \La$ and $t\geq0$.
\item Domination:	$\|D\varphi_t|_{E^{ss}(x)}\|\cdot\|D\varphi_{-t}|_{E^{cu}(\varphi_t(x))}\|\leq C\cdot e^{-\lambda t}$ for all $x\in \La$ and $t\geq0$.
\item Sectional expansion of $E^{cu}$: for any $x\in\La$ and any 2-dimensional linear subspace $P\subset E^{cu}(x)$, one has 
$$|\Det(D\varphi_t|_P)|\geq C^{-1}\cdot e^{\lambda t}~\textrm{ for   $t\geq0$}.$$
\end{itemize}
An invariant compact set is \emph{negatively singular hyperbolic} if it is positively singular hyperbolic with respect to $-X$.  An invariant compact set is \emph{singular hyperbolic} if it is positively or negatively singular hyperbolic.
A vector field is singular hyperbolic if each of its chain recurrence class is either singular hyperbolic or a hyperbolic singularity.
\end{Definition}

The notion of Lorenz like singularity plays an important role.
\begin{Definition}~\label{Def:lorenz-like singularity}
Let $X\in\mathcal{X}^1(M)$.
A singularity $\sigma$ is \emph{Lorenz like}, if $\sigma$ admits a $(D\varphi_t)_{t\in\RR}$-invariant  splitting $T_\sigma M=E^{ss}_\sigma\oplus E^c_\sigma\oplus E^{uu}_\sigma$ with $\dim E^c_\sigma=1$ such that
the largest Lyapunov exponent $\lambda_\sigma^{ss}$ along $E_\sigma^{ss}$, the smallest Lyapunov exponent $\lambda_\sigma^{uu}$ along $E_\sigma^{uu}$ and the center Lyapunov exponent $\lambda_\sigma^c$ along $E^c$ satisfy
$$0<|\lambda_\sigma^c|<\min\{-\lambda_\sigma^{ss},\lambda_\sigma^{uu}\}.$$
\end{Definition}

\begin{Remark}\label{Rem:Lorenz-like-singularity}
Every Lorenz like singularity $\sigma$ is hyperbolic
and admits a strong stable manifold $W^{ss}(\sigma)$  tangent to $E^{ss}_\sigma$ and a strong unstable manifold
$W^{uu}(\sigma)$ tangent to $E^{uu}_\sigma$.
\end{Remark}

Let $X\in\mathcal{X}^1(M)$ and let $\La$ be an  invariant compact set for the flow $(\varphi_t)_{t\in\RR}$. 
Given $\eta,T>0$, a $(\Psi_t)_{t\in\RR}$-invariant splitting  $\cN_{\La\setminus\sing(X)}=\cE\oplus\cF$ is an  \emph{$(\eta, T)$-dominated splitting for the linear Poincar\'e flow}  if 
the spaces $\cE(x)$ (resp. $\cF$) have the same dimension and
$$\|\Psi_T|_{\cE(x)}\|\cdot \|\Psi_{-T}|_{\cF(\varphi_T(x))}\|\leq  e^{-\eta}\textrm{ for any $x\in\La\setminus\sing(X)$}.$$ 
One says that an invariant  compact  set $\Lambda$ admits a dominated splitting for the linear Poincar\'e flow if it admits an $(\eta,T)$-dominated splitting for some $\eta,T>0$.  Similarly, one can define the dominated splitting for the tangent flow $(D\varphi_t)_{t\in\mathbb{R}}$.

The dominated splitting of the linear Poincar\'e flow may not be robust due to the existence of singularities. 
This leads to the notion of \emph{singular domination} defined  in \cite[Definition 1.1]{CYZ2}.

\begin{Definition}\label{d.singular-domination}
Let $X\in\mathcal{X}^1(M)$ and $\Lambda$ be an invariant compact set of $(\varphi_t)_{t\in\mathbb{R}}$. One says that $\Lambda$ admits a \emph{singular domination} if
there exists a splitting $\cN_{\Lambda\setminus\Sing(X)}=\cE\oplus\cF$ such that:
\begin{itemize}
\item the splitting $\cN_{\Lambda\setminus\Sing(X)}=\cE\oplus\cF$ is dominated for the linear Poincar\'e flow $(\Psi_t)_{t\in\mathbb{R}}$;
\item  each $\sigma\in\Lambda\cap\Sing(X)$ satisfies one of the following properties:
\begin{itemize}
\item either  $\sigma$ admits a dominated splitting $T_\sigma M=E^{ss}_\sigma\oplus F_\sigma$ for $(D\varphi_t)_{t\in\mathbb{R}}$,  where $E^{ss}_\sigma$ is uniformly contracting with  $\dim(E^{ss}_\sigma)=\dim(\cE)$  and $W^{ss}(\sigma)\cap \Lambda=\{\sigma \}$;
\item or  $\sigma$ admits a dominated splitting $T_\sigma M=E_\sigma\oplus E^{uu}_\sigma$ for $(D\varphi_t)_{t\in\mathbb{R}}$,  where $E^{uu}_\sigma$ is uniformly expanding  with $\dim(E^{uu}_\sigma)=\dim(\cF)$  and $W^{uu}(\sigma)\cap \Lambda=\{\sigma \}$.
\end{itemize}
\end{itemize}
\end{Definition}

Recall that when $\dim(M)\leq 4$ generic star vector fields are singular hyperbolic~\cite{MPP,SGW}, while when $\dim(M)\geq 5$ da Luz~\cite{dL} constructed generic star vector fields that are not singular hyperbolic. To characterize star vector fields, Bonatti-da Luz~\cite{BdL} introduced the important   notion of   {\it multi-singular hyperbolicity} and proved that star is equivalent to multi-singular hyperbolic open and densely. 
Later on, \cite{CYZ2} gave an alternative definition which is equivalent to the one given by Bonatti-da Luz under a mild condition (see Theorems C and D in~\cite{CYZ2}) and which is based on the notion of singular domination.
This is the definition that we will use here.
\begin{Definition}\label{Def:multi-singular-hyperbolic}
Let $X\in\cX^1(M)$ and $\Lambda$ be an invariant compact set of $(\varphi_t)_{t\in\mathbb{R}}$. One says that $\Lambda$ is  \emph{multi-singular hyperbolic} if either it is a hyperbolic singularity or periodic orbit, or if:
\begin{itemize}
\item $\Lambda$ admits a singular domination $\cN^s\oplus \cN^u$ with $\dim(\cN^s)>0$ and $\dim(\cN^u)>0$;
\item  each $\sigma\in\Lambda\cap\Sing(X)$ is Lorenz like  with a  splitting $T_\sigma M=E^{ss}_\sigma\oplus E^c_\sigma\oplus E^{uu}_\sigma$ satisfying $\dim E^{ss}_\sigma=\dim\cN^s$ and $\dim E^{uu}_\sigma=\dim\cN^u$;
\item there exist a compact isolating neighborhood $V$ of $\sing(X)\cap\Lambda$  and $\eta, T>0$ such that  for any $x\in\Lambda$ and $t\in \mathbb{R}$ satisfying
$x,\varphi_t(x)\notin V$ and  $t>T$,  one has
\begin{equation}~\label{equ:hyperbolicity-away-from-singularity}
\|\Psi_{-t}|_{\cN^u(\varphi_t(x))}\|<e^{-\eta t} \textrm{ and } \|\Psi_t|_{\cN^s(x)}\|<e^{-\eta t}.
\end{equation}
\end{itemize}
\end{Definition}

{
\begin{Lemma}\label{Lem:quasi-attractor-singular-hyperbolic}
A quasi attractor $\Lambda$ which is multi-singular hyperbolic and not a sink, is singular hyperbolic.
\end{Lemma}
\begin{proof}
Since $\Lambda$ is not a sink, it is not reduced to a singularity or to a periodic orbit.
Let us consider a singularity $\sigma$ in $\Lambda$. By Definitions~\ref{Def:lorenz-like singularity} and~\ref{Def:multi-singular-hyperbolic}
it is  Lorenz-like, hyperbolic and admits a splitting $T_\sigma M=
E^{ss}_\sigma\oplus E^c_\sigma\oplus E^{uu}_\sigma$ with $\dim(E^c_\sigma)=1$,
$\dim(E^{ss}_\sigma)=\dim(\cN^s)$ and $\dim(E^{uu}_\sigma)=\dim(\cN^u)$,
where $\cN_{\Lambda\setminus \sing(X)}=\cN^s\oplus \cN^u$
is the multi-singular hyperbolic splitting on $\Lambda$.
Since $\Lambda$ is chain-transitive and not reduced to $\sigma$,
both sets $W^s(\sigma)$ and $W^u(\sigma)$
intersect $\Lambda\setminus \{\sigma\}$
(we say that $\sigma$ is active).

We claim that the center Lyapunov exponent of $\sigma$ is negative.
Otherwise $W^{s}(\sigma)$ has dimension $\dim(E^{ss}_\sigma)$,
so $W^{ss}(\sigma)\cap \Lambda\neq \{\sigma\}$.
Definition~\ref{d.singular-domination} then implies that
$W^{uu}(\sigma)\cap \Lambda=\{\sigma\}$.
But this is a contradiction since $\Lambda$ is a quasi-attractor.

Thus $\Lambda$ only contains Lorenz-like singularities with negative center Lyapunov exponent. In other words, all singularities in $\Lambda$ have the same index.
Then one can conclude by \cite[Theorem C]{CYZ2}.
\end{proof}
}

It is not difficult to see that a hyperbolic set is always singular hyperbolic and a singular hyperbolic set is always multi-singular hyperbolic. The differences appear when the set admits singularities. 
One has the following result from \cite{BdL} and \cite[Theorem C]{CYZ2}. 
\begin{Proposition}\label{Prop:multi-singular-hyperbolic-no-singularity}
Let $X\in\cX^1(M)$ and $\Lambda$ be a multi-singular hyperbolic set.
If $\Lambda\cap \Sing(X)=\emptyset$, then $\Lambda$ is hyperbolic.
\end{Proposition}

A $C^1$ vector field $X$ is called \emph{multi-singular hyperbolic} if each chain recurrence class is a multi-singular hyperbolic set. Denote by ${\cal X}^*(M)$  the space of $C^1$ multi-singular hyperbolic vector fields, endowed the induced topology from ${\cal X}^1(M)$.
\subsection{Properties of hyperbolic vector fields}
Let $\gamma,\gamma^\prime$ be two (non-singular) hyperbolic periodic orbits of $X\in\mathcal{X}^1(M)$. One says  that $\gamma$ and $\gamma^\prime$ are ~\emph{homoclinically related} if the stable manifold of one periodic orbit has non-empty transverse intersection with the unstable manifold of the other, and vice versa. The ~\emph{homoclinic class } of a hyperbolic periodic orbit $\gamma$ is defined as 
$$H(\gamma)=\overline{\bigcup \{\gamma^\prime: \textrm{$\gamma^\prime$ is homoclinically related to $\gamma$} \}}.$$
It is classical that $H(\gamma)$ is  a transitive set (see for instance~\cite{N}).

One summarizes the following well-known result (see for instance~\cite{W16}):

\begin{Proposition}\label{Prop:hyperbolic}
Given $X\in\mathcal{X}^1(M)$,   
each hyperbolic singularity-free chain recurrence class is a homoclinic class and all the periodic orbits contained inside are homoclinically related. 
\end{Proposition}

An attractor contains the unstable manifolds of its hyperbolic periodic points. Conversely, one has the following:
\begin{Proposition}\label{Thm:class-to-attractor}
Let $X\in\mathcal{X}^1(M)$ and $\Lambda$ be a hyperbolic chain recurrence class which is singularity-free.
If $\Lambda$ contains the unstable manifold of a hyperbolic periodic orbit $\gamma$, then $\Lambda$ is an attractor.
\end{Proposition}
\begin{proof}
By Proposition~\ref{Prop:hyperbolic}, $\Lambda=H(\gamma)$, and as a consequence $\Lambda$ is transitive.
By definition, $H(\gamma)\subset \overline{W^u(\gamma)}$, and  by the assumption $\overline{W^u(\gamma)}\subset \Lambda$, one  has $\Lambda=H(\gamma)=\overline{W^u(\gamma)}$. 
	
Since $\Lambda$ is hyperbolic and coincides with $\overline{W^u(\gamma)}$,
the shadowing lemma implies that for any two small neighborhoods $U,V$ of $\Lambda$
there exists $T>0$ such that $\varphi_t(U)\subset V$ for any $t>T$.
This implies that $\Lambda$ is attracting, hence is an attractor.
\end{proof}

The following is a classical result proved in~\cite[Theorem 5.6]{BR}. 
\begin{Proposition}\label{Prop:lebesgue-basin}
If $X$ is a $C^2$ vector field, then each hyperbolic singularity-free chain recurrence class $\Lambda$
is either an attractor or satisfies ${\rm Leb}(\{x:\omega(x)\subset\Lambda\})=0$.	
\end{Proposition}
\begin{Remark}
One can provide a proof of this proposition in another way: by Proposition~\ref{Prop:hyperbolic}, the set $\Lambda$ is a hyperbolic homoclinic class hence it is transitive; if $\Leb(\{x:\omega(x)\subset\Lambda\})>0$, by Theorem F in \cite{CYZ}	and Theorem A in~\cite{LY}, there exists a point $x\in \Lambda$ such that the unstable manifold $W^{u}(x)$ of $x$ is contained in $\Lambda$. Since $\Lambda$ is hyperbolic and transitive, the set $\Lambda$ is saturated by unstable manifolds, hence it is an attractor.
\end{Remark}

\subsection{Number of chain recurrence classes}

\begin{Proposition}\label{Prop:contable-class}
Multi-singular hyperbolic vector fields have at most countably many chain recurrence classes.
\end{Proposition}
\begin{proof}
Since the singularities are hyperbolic, there exist at most finitely many singularities; thus, there are only finitely many chain recurrence classes containing singularities.
Any singularity-free chain recurrence class is hyperbolic by Proposition~\ref{Prop:multi-singular-hyperbolic-no-singularity}, hence contains a hyperbolic periodic orbit. Since a  vector field has at most countably many hyperbolic periodic orbits, it has at most countably many singularity-free chain recurrence classes.
\end{proof}

Recall that every multi-singular hyperbolic vector field satisfies the star condition: no non-hyperbolic critical element can be created by $C^1$ small perturbations~\cite{BdL}.
For star vector fields, S. Liao~\cite{L} proved the finiteness of the set of sinks (see also~\cite[Theorem A]{YZ}). Thus one has the following theorem. 
\begin{Theorem}~\label{thm.finite-sink}
Multi-singular hyperbolic vector fields have at most finitely many sinks.
\end{Theorem}

\subsection{Lyapunov exponents and hyperbolic measures}
Given an invariant measure, Oseledec theorem~\cite{O} associates at almost every point $x$
some numbers $\lambda_i(x)$ called \emph{Lyapunov exponents of $x$}.
\begin{Theorem}\label{Thm:oseledec}
Let $X$ be a $C^1$ vector field and $\mu$ be an invariant measure. There is a $(\varphi_t)_{t\in\mathbb{R}}$-invariant measurable  set $\Gamma$ with $\mu(\Gamma)=1$, such that for any $x\in\Gamma$, there are $k=k(x)$ numbers $\lambda_1(x)<\lambda_2(x)<\cdots<\lambda_k(x)$ and a  splitting $T_xM=E_1(x)\oplus E_2(x)\oplus\cdots\oplus E_k(x)$
satisfying:
\begin{itemize}		
\item the splitting is invariant:
$D\varphi_t(E_i(x))=E_i(\varphi_t(x))$ for any $x\in\Gamma$, any $1\le i\le k$ and $t\in\RR$; 
\item $\lim_{t\to\pm\infty}\frac{1}{t}\log\|D\varphi_t(v)\|=\lambda_i(x)$, for any $x\in \Gamma$, $1\le i\le k$ and any unit vector $v\in E_i(x)$.
\end{itemize}
If $\mu$ is ergodic, then $k$ and $\lambda_1,\cdots,\lambda_k$ are constant $\mu$-almost everywhere.
\end{Theorem}

An invariant measure is \emph{regular} if $\mu(\Sing(X))=0$. A regular ergodic measure is said to be \emph{hyperbolic} if it admits only one zero Lyapunov exponent (along the direction $\RR\cdot X$).

The \emph{index} $\ind(\mu)$ of a hyperbolic ergodic measure $\mu$ is defined as the number of its negative Lyapunov exponents (counting the multiplicities), i.e. the dimension of the space $\bigoplus\limits_{\lambda_i<0} E_i$.
The index $\ind(\gamma)$ of a hyperbolic critical element $\gamma$
is defined as the index of the unique invariant measure supported on $\gamma$.
This coincides with the index of hyperbolic sets defined in Remark~\ref{Rem:hyperbolicity}.

\begin{Definition}\label{Def:positive-sum}
Let $X\in\mathcal{X}^1(M)$ and let $\mu$ be an invariant measure of  $(\varphi_t)_{t\in\mathbb{R}}$.
For $\mu$-almost every point $x$,
the \emph{sum of its positive Lyapunov exponents} is defined as
$$\sum\lambda^+(x):=\sum_{1\le i\le k}\max(\lambda_i(x),0)\cdot\dim E_i(x).$$
If $\mu$ is ergodic, $\sum\lambda^+(x)$ is constant $\mu$-almost everywhere:
one sets $\sum\lambda^+(\mu)=\sum\lambda^+(x)$.
\end{Definition}

Under a dominated splitting, the following property holds.
\begin{Proposition}\label{Pro:bundle-integral}
Let $X\in\mathcal{X}^1(M)$ and let $\mu$ be a regular ergodic measure of $(\varphi_t)_{t\in\mathbb{R}}$
satisfying:
\begin{itemize}
\item  
there is a dominated splitting $\cN_{\supp(\mu)\setminus\sing(X)}=\cE\oplus\cF$
for the linear Poincar\'e flow; 
\item  the number of positive Lyapunov exponents of $\mu$ equals  $\dim\cF$.
\end{itemize}
Then $$\sum \lambda^+(\mu)=\int\log|\Det(D\varphi_1|_{\cF\oplus \RR\cdot X})|{\rm d}\mu=\int\log|\Det (\Psi_1|_{\cF})|{\rm d}\mu.$$
\end{Proposition}
\begin{proof}
The direct sum of all bundles with positive Lyapunov exponents are contained in $\cF\oplus \RR\cdot X$. Since   the flow direction provides zero Lyapunov exponents,
$$\sum \lambda^+(\mu)=\int\log|\Det(D\varphi_1|_{\cF\oplus \RR\cdot X})|{\rm d}\mu.$$
On the other hand, by the definition of the linear Poincar\'e flow $(\Psi_t)_{t\in\mathbb{R}}$,
$$\int\log|\Det(D\varphi_1|_{\cF\oplus \RR\cdot X})|{\rm d}\mu=\int\left(\log|\Det(\Psi_1|_{\cF})| + \log \|D\varphi_1|_{\mathbb{R}\cdot X}\|\right){\rm d}\mu.$$
Using again $\displaystyle\int \log \|D\varphi_1|_{\mathbb{R}\cdot X}\|{\rm d}\mu=0$, one gets
$$\int\log|\Det(D\varphi_1|_{\cF\oplus \RR\cdot X})|{\rm d}\mu=\int\log|\Det(\Psi_1|_{\cF})|{\rm d}\mu.$$
This completes the proof.
\end{proof}

In the special case of a multi-singular hyperbolic set, Proposition~\ref{Pro:bundle-integral} gives:
\begin{Proposition}\label{Prop:sum-positive-jacobian}
Let $X\in\mathcal{X}^1(M)$ and   $\Lambda$ be a multi-singular hyperbolic set with the splitting $\cN_{\Lambda\setminus{\rm Sing}(X)}=\cN^s\oplus\cN^u$. Then any regular ergodic measure $\mu$ supported on $\Lambda$ is hyperbolic and
\begin{equation}\label{e.entropy}
\sum\lambda^+(\mu)=\int\log|\Det(\Psi_1|_{\cN^{u}})|\ud\mu.
\end{equation}
For a positively singular hyperbolic set $\Lambda$, 
\eqref{e.entropy} becomes:
$\sum\lambda^+(\mu)=\int\log|\Det(D\varphi_1|_{E^{cu}})|{\rm d}\mu.$
\end{Proposition}

\subsection{Invariant manifolds and homoclinic classes}
Given a point $x$, its stable set is defined to be 
$$W^{s}(x)=\big\{y\in M:~\lim_{t\to+\infty}d(\varphi_t(x),\varphi_t(y))=0\big\}.$$
This stable set is not enough in the study of flow due to the existence of shears. One has to define the stable set  of an orbit: given a point $x$, define
$$W^s({\rm Orb}(x))=\big\{y\in M:~\exists \text{~a homeomorphism~} \theta:~\RR^+\to\RR^+~\textrm{s.t.}~\lim_{t\to+\infty}d(\varphi_t(x),\varphi_{\theta(t)}(y))=0\big\}.$$
When $X$ is $C^2$ and $\mu$ is a hyperbolic measure,
Pesin theory associates to $\mu$-almost every point $x$
a stable manifold which is tangent to the sum of the Oseledets spaces $E_i(x)$
associated to negative Lyapunov exponents and which coincides with the stable set $W^s(x)$;
moreover the stable set of the orbit $\gamma$ of $x$ is the union of stable sets of points in $\gamma$:
$$W^s(\gamma)=\bigcup_{y\in\gamma}W^{s}(y).$$
This applies in particular to hyperbolic periodic orbits.
Analogously, one defines the unstable set or manifold of points $W^u(x)$ and orbits $W^u(\gamma)$.

We say that a hyperbolic ergodic measure $\mu$ and a hyperbolic periodic orbit $\gamma$ are
\emph{homoclinically related} if for $\mu$-almost every point $x$,
the stable manifold $W^s(x)$ has non-empty transverse intersections with the unstable manifold of $\gamma$,
and the unstable manifold $W^u(x)$ has non-empty transverse intersections with the stable manifold of $\gamma$.
	
The next result is the flow version of the classical result of Katok \cite{K} whose proof can be found in~\cite{LL} (see also \cite[Theorem 5.6]{SGW} under star condition). In \cite{LL},  the authors carefully combine Pesin theory with Liao's theory on scaled tubular neighborhoods
due to a lack of uniform continuity caused by the singularities. One can also employee \cite[Section 2.3]{PYY} and Liao's shadowing lemma proven by Gan~\cite{G} to prove the following result.
\begin{Theorem}\label{Thm:hyperbolic-measure-ergodic-class}
Let $\mu$ be a regular hyperbolic  ergodic measure of a $C^2$ vector field $X$. 
Then there is a hyperbolic  periodic orbit $\gamma$ such that
\begin{itemize}
\item	$\supp(\mu)\subset H(\gamma)$.
\item $\mu$ is homoclinically related with $\gamma$; in particular, for $\mu$-almost every point $x$, the iterates $\varphi_t(W^u(x))$ accumulate on $W^u(\gamma)$ when $t\to+\infty$.
\end{itemize}
\end{Theorem}

\subsection{Metric entropy and SRB measures}
Given a diffeomorphism $f$ on $M$ and an $f$-invariant measure $\mu$,
one associates the Kolmo\-go\-rov-\-Sina\"{\i}  metric entropy ${\rm h}_\mu(f)$. 
Given the flow $(\varphi_t)_{t\in\mathbb{R}}$ on $M$ and  a $(\varphi_t)_{t\in\mathbb{R}}$-inva\-riant measure $\mu$, the metric entropy of the flow $(\varphi_t)_{t\in\mathbb{R}}$ is defined as the entropy ${\rm h}_\mu(\varphi_1)$ of the time one map $\varphi_1$. Indeed, one has that ${\rm h}_\mu(\varphi_t)=|t|\cdot {\rm h}_\mu(\varphi_1)$ for each $t\in\RR$ (see~\cite[\S6,Theorem 3]{CFS}).
\medskip

We recall Ruelle's inequality \cite{R}.
\begin{Theorem}\label{Thm:Ruelle-inequality}
Let $X\in\mathcal{X}^1(M)$ and assume that $\mu$ is an invariant measure of $(\varphi_t)_{t\in\mathbb{R}}$.
Then,
$${\rm h}_{\mu}(\varphi_1)\le \int\sum\lambda^+(x)\ud\mu.$$
\end{Theorem}

An ergodic measure $\mu$ is \emph{SRB} if it has at least one positive Lyapunov exponent
and ${\rm h}_\mu(\varphi_1)=\sum\lambda^+(\mu)$.
A general invariant measure is SRB if almost all its ergodic components are SRB.

Ledrappier and Young have shown~\cite[Theorem A]{LY} that the above definition of an SRB measure 
is equivalent to another one, which is defined  more geometrically and  is usually taken as the definition
of an SRB measure. Here we state the version for vector fields.

\begin{Theorem}\label{Thm:Ledrappier-Young}
Let $X$ be a $C^2$ vector field  and $\mu$ be an ergodic measure exhibiting positive Lyapunov exponents. Then $\mu$  is SRB if and only if $\mu$ admits positive Lyapunov exponents and the conditional measures of $\mu$ along Pesin unstable manifolds are absolutely continuous with respect to Lebesgue measure.
\end{Theorem}

The following consequence is similar to \cite[Proposition D]{T}.
	
\begin{Corollary}\label{Cor:SRB-ergodic-class}
Given a $C^2$ vector field $X$ and a hyperbolic SRB measure $\mu$, there is a hyperbolic  periodic orbit $\gamma$ such that
		$W^u(\gamma)\subset \supp(\mu)$.
\end{Corollary}
\begin{proof}
Let $\gamma$ be a hyperbolic periodic orbit given by Theorem~\ref{Thm:hyperbolic-measure-ergodic-class}
and $x$ be a typical point of $\mu$. Since $\mu$ is  SRB, by Theorem~\ref{Thm:Ledrappier-Young},    Lebesgue almost every point in the unstable manifold $W^u(x)$ of $x$ is typical for $\mu$. This implies that $W^u(x)\subset\supp(\mu)$. By Theorem~\ref{Thm:hyperbolic-measure-ergodic-class}, $W^u(\gamma)$ is accumulated by $\varphi_t(W^u(x))$ when $t\to\infty$. Thus, $W^u(\gamma)$ is contained in $\supp(\mu)$.
\end{proof}

\subsection{Singular hyperbolic attractors}
An invariant compact set $\Lambda$ is \emph{robustly transitive} for $X$
if it admits a neighborhood $U$ such that $\Lambda=\bigcap _{t\in \RR} \varphi^X_t(U)$
and for any vector field $Y$ that is $C^1$ close to $X$, the set
$\Lambda^Y=\bigcap _{t\in \RR} \varphi^Y_t(U)$ is transitive for $Y$.
\begin{Theorem}[\cite{CY}]~\label{thm.continuity-of-singular-hyperbolic-attractor}
There exists an open and dense subset $\cU$ of $\cX^1(M)$ such that for any $X\in\cU$,
if $\La$ is a singular hyperbolic Lyapunov stable chain recurrence class,  then $\Lambda$ is a robustly transitive attractor; moreover $\Lambda=H(\gamma)$ for any periodic orbit $\gamma\subset \Lambda$.
\end{Theorem}

\begin{Remark}\label{Rem:singular-hyperbolic-attractor}
A singular hyperbolic attractor is always positively singular hyperbolic, see for instance~\cite[Proposition 2.4]{CY}.
\end{Remark}

Under singular hyperbolicity, \cite{PYY} shows that the entropy is upper semi-continuous:

\begin{Theorem}\label{Thm;ups-entropy}
Let $\Lambda$ be a singular hyperbolic attractor of a $C^1$ vector field  $X$. Assume that $\{X_n\}$ is a sequence of vector fields such that $\lim_{n\to\infty}X_n=X$ in $\cX^1(M)$ and that each $X_n$ has an invariant measure $\mu_n$ such that $\mu=\lim_{n\to\infty}\mu_n$ is supported on $\Lambda$. Then,
$${\rm h}_{\mu}(\varphi_1^X)\ge\limsup_{n\to\infty}{\rm h}_{\mu_n}(\varphi_1^{X_n}).$$
\end{Theorem}
Theorem~\ref{Thm;ups-entropy} is based on \cite[Theorem A]{PYY} which asserts that if $\Lambda$ is a singular hyperbolic set for a $C^1$ vector field, then there is $\delta>0$ such that on a neighborhood of  $\Lambda$, the dynamics is robustly $\delta$-entropy expansive. Then as commented in \cite{PYY}, a robust uniform entropy expansiveness implies the upper semi-continuity of the entropy as in Theorem~\ref{Thm;ups-entropy}.

\subsection{Generic properties of multi-singular hyperbolic vector fields}
The following theorem states some known generic properties.

\begin{Theorem}~\label{thm.generic-property-of-star}
There exists a dense G$_\delta$ subset $\cG\subset\cX^1(M)$ such that any $X\in\cG$ satisfies the following properties.
\begin{enumerate}
\item\label{i.attractor-star}    Any multi-singular hyperbolic  Lyapunov stable chain recurrence class of $X$ is an attractor; more precisely, it is  either a sink, or a singular hyperbolic attractor. 
\item\label{i.lyapunov-stable-star} If a multi-singular hyperbolic chain recurrence class $C$ contains the unstable manifold of a critical element, then $C$ is a singular hyperbolic attractor.
\end{enumerate}
\end{Theorem}
\begin{Remark}
The first item  comes from Corollary E in \cite{PYY}.  Combining the first item with  the connecting lemma for pseudo orbits in \cite{BC}, one gets the second item. 	
\end{Remark}

\section{Pliss points and the intersections of invariant manifolds}

\subsection{Pliss points}
For a periodic orbit $\gamma$, we denote by $\pi(\gamma)$ its period.
Liao has proved the following fundamental  property of star vector fields (with singularities or not).
We will apply it to multi-singular hyperbolic vector fields.

\begin{Theorem}[\cite{L79}]\label{Thm:fundamental}
	For any  $X\in\cX^*(M)$,  there exist $\eta,T>0$ and a $C^1$ neighborhood $\cU$ of $X$ with the following properties. For any $Y\in\cU$ and  any periodic orbit $\gamma$ of $Y$ with $\pi(\gamma)\geq T$:
\begin{itemize}
\item
the hyperbolic splitting $\cN_{\gamma}=\cN^s\oplus\cN^u$ is $(2\eta,T)$-dominated,
\item $\prod_{i=0}^{[\pi(\gamma)/T]-1}\|\Psi^Y_T|_{\cN^s(\varphi^Y_{iT}(p))}\|\leq e^{-\eta\pi(\gamma)} \textrm{\:~and~\:}\prod_{i=0}^{[\pi(\gamma)/T]-1}\|\Psi^Y_{-T}|_{\cN^u(\varphi^Y_{-iT}(p))}\|\leq e^{-\eta\pi(\gamma)}$
for each $p\in \gamma$.
\end{itemize}	
\end{Theorem}

One also considers a \emph{rescaled linear Poincar\'e flow $(\Psi_t^*)_{t\in\mathbb{R}}$}: given any regular point $x$, a vector $v\in\cN_x$ and $t\in\mathbb{R}$, one defines
$$\Psi_t^*(v)=\frac{\Psi_t(v)}{\|D\varphi_t|_{\RR\cdot X(x)}\|}.$$

\begin{Remark}~\label{r.fundamental-property}
Theorem~\ref{Thm:fundamental} also holds for the rescaled linear Poincar\'e flow $(\Psi_t^*)_{t\in\RR}$.
\end{Remark}

Let $X$ be a $C^1$ vector field on $M$ and $\Lambda$ be an invariant compact set.
Assume that $\Lambda$ admits a dominated splitting $\cN_{\Lambda\setminus\sing(X)}=\cE\oplus\cF$.
Let $\eta,T>0$ and $x\in\La\setminus\sing(X)$.

We say that $x$  is an \emph{$(\eta,T,\cE)$-Pliss point} for the flow $(\Psi_t^*)_{t\in\RR}$ if it satisfies
$$\prod_{i=0}^{n-1}\|\Psi_T^*|_{\cE(\varphi_{iT}(x))}\|\leq e^{-n\eta}\textrm{ for all $n\geq 1$}.$$
We define similarly the notion of \emph{$(\eta,T,\cF)$-Pliss point}
by considering the backward orbit of $x$.

The  point $x$ is an \emph{$(\eta,T)$-bi-Pliss point} for $(\Psi^*_t)_{t\in\RR}$ (with respect to the splitting $\cE\oplus\cF$) if $x$ is both an $(\eta,T,\cE)$-Pliss point and an $(\eta,T,\cF)$-Pliss point.

We present a criterion for the existence of bi-Pliss points, proved in~\cite[Page 990]{PS}.
(The criterion there is stated for maps. In our case we consider the diffeomorphism $\varphi_T$.)

\begin{Lemma}~\label{l.bi-pliss}
For $X\in\cX^1(M)$, let $\La$ be an invariant compact set and $\eta,T>0$. Assume that  $\Lambda$ admits a $(2\eta,T)$-dominated splitting $\cN_{\La\setminus\sing(X)}=\cE\oplus \cF$ for  $(\Psi^*_t)_{t\in\RR}$. 	
If the forward orbit of some $x\in\La\setminus\sing(X)$ under
the map $\varphi_T$ contains an $(\eta, T, \cE)$-Pliss point and its backward orbit under
$\varphi_T$ contains an $(\eta,T,\cF)$-Pliss point for  $(\Psi^*_t)_{t\in\RR}$, then there exists an $(\eta,T)$-bi-Pliss point $y$ for $(\Psi^*_t)_{t\in\RR}$ in the orbit of $x$ under  $\varphi_T$.
\end{Lemma}

\begin{Corollary}\label{Lem:bi-Pliss-exsitence}
Let $X\in\cX^1(M)$ be a multi-singular hyperbolic vector field.
There exist $T,\eta>0$ and a $C^1$ neighborhood $\cU$ of $X$ with the following properties. For any $Y\in\cU$ and  any periodic orbit $\gamma$ of $Y$ with $\pi(\gamma)\geq T$,  let $\cN_\gamma=\cN^{s}\oplus\cN^{u}$ be its hyperbolic splitting for $(\Psi_t)_{t\in\RR}$. Then there exists an $(\eta,T)$-bi-pliss point $x\in\gamma$
for  $(\Psi_t^*)_{t\in\RR}$ with respect to $\cN_\gamma=\cN^{s}\oplus\cN^{u}$.	
\end{Corollary}
\begin{proof}
Since $X$ is multi-singular hyperbolic, it is a star vector field.
Then Theorem~\ref{Thm:fundamental} and Remark ~\ref{r.fundamental-property}
give $\eta_0,T>0$
and a $C^1$ neighborhood $\cU$ of $X$ such that
for any periodic orbit $\gamma$ of any $Y\in \cU$,
the hyperbolic splitting $\cN_\gamma=\cN^{s}\oplus\cN^{u}$
is $(2\eta_0,T)$-dominated for $(\Psi^*_t)_{t\in\RR}$ and
$$\limsup_{n} \frac 1 n \sum_{i=0}^{n-1} \log \|\Psi^{*,Y}_T|_{\cN^s(\varphi^Y_{iT}(p))}\| \leq -\eta_0
\text{ \;and\; }
\limsup_{n} \frac 1 n \sum_{i=0}^{n-1} \log \|\Psi^{*,Y}_{-T}|_{\cN^u(\varphi^Y_{-iT}(p))}\| \leq -\eta_0.$$
Let $\eta\in (0,\eta_0)$ and $p\in \gamma$.
Note that $\|\Psi^*_T\|$ is bounded~\cite{GY}.
Pliss lemma~\cite{P} gives a forward iterate of $p$ for $\varphi_T$ which is
a $(\eta,T,\cN^s)$-Pliss point for $(\Psi^*_t)_{t\in\RR}$ and a backward iterate of $p$ for $\varphi_T$ which is
a $(\eta,T,\cN^u)$-Pliss point for $(\Psi^*_t)_{t\in\RR}$.
One concludes by Lemma~\ref{l.bi-pliss}.
\end{proof}

The following proposition gives the existence of bi-Pliss periodic points in a given region provided that one can get bi-Pliss periodic points by small perturbations. 
\begin{Proposition}\label{Pro:bi-Pliss}
Given  $\eta,T>0$, there exists a dense G$_\delta$ subset $\cG$ of $\cX^1(M)$ such that for any $X\in\cG$, the following properties are satisfied.
Let us assume that there exist a sequence of vector fields $\{X_n\}_{n\in\NN}$ converging to $X$,
a sequence $\{p_n\}_{n\in\NN}$ in $M$ converging to some point $z$ and numbers $\eta_n>\eta$ satisfying:
\begin{itemize}
\item $p_n$ is a hyperbolic periodic point for $X_n$ which is $(\eta_n,T)$-bi-Pliss for $(\Psi_t^{*,X_n})_{t\in\mathbb{R}}$,
\item the  hyperbolic splitting $\cN_{\orb(p_n)}=\cN^{s}\oplus \cN^{u}$ for $(\Psi^{*,X_n}_t)_{t\in\RR}$ is $(\eta_n,T)$-dominated.
\end{itemize}
Then for any neighborhood $U$ of $z$ there exist $p\in U$ and $\eta'\ge\eta$ satisfying:
\begin{itemize}
\item $p$ is a hyperbolic periodic point for $X$ which is $(\eta',T)$-bi-Pliss for $(\Psi_t^{*,X})_{t\in\mathbb{R}}$,
\item the hyperbolic splitting $\cN_{\orb(p)}=\cN^{s}\oplus \cN^{u}$ for $(\Psi^{*,X}_t)_{t\in\RR}$ is $(\eta',T)$-dominated.
\end{itemize}
\end{Proposition}

\begin{proof} 
Fix a countable basis $\{O_i\}_{i\in\mathbb{N}}$ of non-empty open sets of $M$ and let $\{U_k\}_{k\in\mathbb{N}}$  be the countable family consisting of all finite unions of elements in $\{O_i\}_{i\in\mathbb{N}}$. 
	
For any $1\le i\le \dim M-2$ and $k \in\NN$, let 
$\cH^i_{k}\subset \cX^1(M)$ be the set of vector fields $X$ such that 
there exist $\eta'>\eta$ and a hyperbolic periodic point $p\in U_k$ of index $i$ whose hyperbolic splitting $\cN_{{\rm Orb}(p)}=\cN^s\oplus\cN^u$ is $(\eta',T)$-dominated for $(\Psi_t^{*,X})_{t\in\mathbb{R}}$ and which is $(\eta',T)$-bi-Pliss for $(\Psi_t^{*,X})_{t\in\mathbb{R}}$.
By the robustness of the domination over hyperbolic periodic orbits, $\cH^i_{k}$ is an open set.
Let $\cN^i_{k}$ be the complement of the closure of $\cH^i_{k}$ in $\cX^1(M)$. By definition, $\cN^i_{k}$ is an open set and $\cN^i_{k}\cup \cH^i_{k}$ is dense in $\cX^1(M)$.
	
We define the following dense G$_\delta$ subset of $\cX^1(M)$:
$$\cG=\bigcap_{i=1}^{\dim(M)-2}\bigcap_{k\in\NN}(\cH_{k}^i\cup\cN_{k}^i).$$
We now can check $X\in\cG$ satisfies the required properties.
Let us consider $X_n\to X$ and $p_n\to z$ and $\eta_n>\eta$ as in the statement of the proposition
and let us fix a small neighborhood $U_k$ of $z$. Taking a subsequence, we can assume that
all the periodic points $p_n$ have the same index $i$.
This shows that $X$ belongs to the closure of $\cH^i_{k}$. Since $X$ belongs to $\cG$,
it belongs to the open set $\cH^i_{k}$, hence satisfies the conclusion of the proposition.
\end{proof}

The following result shows that the set of  Pliss points is closed. The proof goes just by taking a limit and by continuity, hence is omitted.
\begin{Proposition}\label{Pro:limit-bi-pliss}
Let $X\in\cX^1(M)$ and  assume that  there exist $\eta>0$, $T>0$,  a sequence of points $\{x_n\}$ and a sequence of $C^1$ vector fields $\{X_n\}$ satisfying
	\begin{itemize}
		\item $\lim\limits_{n\to\infty}X_n=X$, $\lim\limits_{n\to\infty}x_n=x$ is a regular point of $X$;
		\item for each $n$, there is an $(\eta,T)$-dominated splitting $\cN_{{\rm Orb}(x_n,X_n)}=\cE_n\oplus\cF_n$ for $(\Psi_t^{X_n})_{t\in\mathbb{R}}$.
	\end{itemize}
	Then there is an $(\eta,T)$-dominated splitting $\cN_{{\rm Orb}(x,X)}=\cE\oplus\cF$  for $(\Psi_t^{X})_{t\in\mathbb{R}}$ satisfying:
	\begin{itemize}
		\item If $x_n$ is an  $(\eta,T,\cE_n)$-Pliss point of $(\Psi_t^{*,X_n})_{t\in\mathbb{R}}$ for each $n$, then $x$ is an  $(\eta,T,\cE)$-Pliss point of $(\Psi_t^{*,X})_{t\in\mathbb{R}}$.
		\item If $x_n$ is an  $(\eta,T,\cF_n)$-Pliss point of $(\Psi_t^{*,X_n})_{t\in\mathbb{R}}$ for each $n$, then $x$ is an  $(\eta,T,\cF)$-Pliss point of $(\Psi_t^{*,X})_{t\in\mathbb{R}}$.
		\item If $x_n$ is an  $(\eta,T)$-bi-Pliss point of $(\Psi_t^{*,X_n})_{t\in\mathbb{R}}$ for each $n$, then $x$ is an  $(\eta,T)$-bi-Pliss point of $(\Psi_t^{*,X})_{t\in\mathbb{R}}$.
	\end{itemize}
\end{Proposition}

\subsection{Stable manifolds and stable sets of orbits}
By \cite[Theorem B]{CYZ2},
for any $X\in \cX^1(M)$ and
any invariant compact set $\Lambda$ admitting a singular dominated splitting
$\cN_{\Lambda\setminus{\rm Sing}(X)}=\cE\oplus\cF$,
there exist a neighborhood $U$ of $\Lambda$ and a $C^1$ neighborhood $\cU$ of $X$
such that any invariant compact set $\Lambda^Y$ in $U$ for any $Y\in \cU$
admits a singular dominated splitting
$\cN_{\Lambda^Y\setminus{\rm Sing}(Y)}=\cE^Y\oplus\cF^Y$
such that $\dim(\cE^Y)=\dim(\cE)$ and $\dim(\cF^Y)=\dim(\cF)$.
The following has been essentially proved in~\cite{GY}.

\begin{Proposition}\label{Pro:size-stable-Pliss}
For $X\in\cX^1(M)$, let $\La$ be an invariant compact set
with a singular dominated splitting $\cN_{\Lambda\setminus{\rm Sing}(X)}=\cE\oplus\cF$.
Given $\eta>0$ and $T>0$, there are $\delta>0$, a neighborhood $U$ of $\Lambda$ and a $C^1$ neighborhood $\cU$ of $X$ such that for any $Y\in\cU$ and any $(\eta,T,\cE^Y)$-Pliss point $p$ for $(\Psi_t^{*,Y})_{t\in\mathbb{R}}$ satisfying ${\rm Orb}(p,Y)\subset U\setminus{\rm Sing}(Y)$,
\begin{enumerate} 
\item there is a $C^1$ embedding
$\kappa_p\colon \cE^Y_p(\delta\|Y(p)\|)\to M$
such that
$$\kappa_p(0)=p,\quad D_0\kappa_p\cdot\cE^Y_p=\cE^Y_p$$
and the image  $W^{\cE^Y}_{\delta\|Y(p)\|}(p,Y):={\rm Im}(\kappa_p)$
is contained in the stable set $W^s({\rm Orb}(p,Y))$.

\item\label{plaque2} for any sequences $Y_n\to Y$ in $\cX^1(M)$ and $p_n\to p$ in $M$
such that each $p_n$ is an $(\eta,T,\cE^{Y_n})$-Pliss point   for $(\Psi_t^{*,Y_n})_{t\in\mathbb{R}}$ satisfying ${\rm Orb}(p_n,Y_n)\subset U\setminus{\rm Sing}(Y_n)$,
the sequence of submanifolds $W^{\cE^{Y_n}}_{\delta\|Y_n(p_n)\|}(p_n,Y_n)$
converges towards $W^{\cE^Y}_{\delta\|Y(p)\|}(p,Y)$ for the $C^1$ topology.
 \end{enumerate}
An analogous property defines submanifolds $W^{\cF^Y}_{\delta\|Y(p)\|}(p,Y)$
at $(\eta,T,\cF^Y)$-Pliss points.

Moreover any $z\in \Lambda\setminus {\rm Sing}(X)$ admits a neighborhood $V_z$ such that,
for any $p,q\in V_z$ with $X$-orbits in $U\setminus{\rm Sing}(X)$,
if $p$ is $(\eta,T,\cE)$-Pliss for $(\Psi_t^{*,X})_{t\in\RR}$ and
$q$ is $(\eta,T,\cF)$-Pliss for $(\Psi_t^{*,X})_{t\in\RR}$,
$$\varphi_{(-1,1)}(W^{\cE}_{\delta\|X(p)\|}(p,X))\pitchfork W^{\cF}_{\delta\|X(q)\|}(q,X)\neq\emptyset.$$
\end{Proposition}
\begin{proof}[Comments about the proof]
We follow the proof of~\cite{GY}.
We build a non-linear sectional Poincar\'e flow on a neighborhood of the zero-section
of the normal bundle $\cN_U$:
for any regular point $x\in \Lambda$
and any time $t$, the holonomy associates to the flow defines a map
from a neighborhood of $x$ in $\exp_x(\cN)$ to $\exp_{\varphi_t(x)}(\cN)$;
conjugating by the exponential maps, one obtains a local diffeomorphism
$\cP_{x,\varphi_t(x)}$ between neighborhoods of $0$ in $\cN_x$
and $\cN_{\varphi_t(x)}$. By rescaling
by $\|X(x)\|$, one defines local diffeomorphisms
$\cP^*_{x,\varphi_t(x)}$ which are defined on uniform balls $\cN_x(\beta)$
for times $t\in [-1,1]$. See~\cite[Section 2.2]{GY}.

We then apply the plaque family theorem~\cite{HPS}
to the sequence of local diffeomorphisms $\cP^*_{\varphi_{n}(x),\varphi_{n+1}(x)}$
and to the dominated splitting $\cE\oplus \cF$
and get a continuous family of locally invariant uniform $C^1$ plaques $\Delta^\cE_x\subset \cN_x$
tangent to $\cE_x$ at $0$.
Its projection by the exponential map $\exp_x$
is a submanifold of $M$ tangent to $\cE_x$ at $x$.
The local invariance of the plaques and the Pliss property imply
that when $x$ is a $(\eta,T,\cE)$-Pliss point, the projection $\exp_x(\Delta^\cE_x)$
is contained in the stable set of the orbit of $x$.
By definition of the rescaling, the projection has size proportional to $\|X(x)\|$.
Hence there exists $\delta>0$ such that
the projection contains a ball of radius $\delta \|X(x)\|$,
that we denote by $W^\cE_{\delta\|X(x)\|}(x)$.
The continuity of the plaque family implies the continuity of the submanifolds
$W^\cE_{\delta\|X(x)\|}(x)$ with respect to $x$ for the $C^1$ topology, as stated in
Item~\ref{plaque2} for a fixed vector field $X$.

Note that one can extend the definition of (non-linear) rescaled sectional Poincar\'e flow
for vector fields $Y$ that belong to a $C^1$ neighborhood $\cU$ of $X$
and regular $Y$-orbits contained in a neighborhood $U$ of $\Lambda$.
The plaque family theorem gives plaques $\Delta^{\cE^Y}_x$
that vary continuously with $x$ and $Y$.
After projection to $M$, one gets (at Pliss points)
projections $W^{\cE^Y}_{\delta\|Y(x)\|}(x)$
which vary continuously with $x$ and $Y$ for the $C^1$ topology.

The last property of the proposition is~\cite[Corollary 2.16]{GY}.
\end{proof}

\subsection{Intersections between invariant manifolds of periodic orbits and singularities}
The following theorem is stated in~\cite[Theorem 3.10]{GYZ} and generalizes a result by Liao~\cite{L}
(for sequences of uniform sinks accumulating on a singularity). It is proved in~\cite{Z}, see also~\cite[Section 3]{PYY2}.   

\begin{Theorem}\label{Thm:intersection-singular-periodic}
For $X\in\mathcal{X}^1(M)$, let $\sigma$ be a Lorenz like singularity with negative center Lyapunov exponent.
Let $\{X_n\}_{n\in\NN}$ be a sequence of vector fields with hyperbolic periodic points $p_n$ such that:
	\begin{itemize}
		\item $X_n\to X$ in $\cX^1(M)$ and $p_n\to \sigma$.
		\item There are $\eta,T>0$ such that the hyperbolic splitting $\cN_{{\rm Orb}(p_n,X_n)}=\cN^s\oplus \cN^u$ is $(\eta,T)$-dominated for $(\Psi_t^{*,X_n})_{t\in \RR}$
		and $\ind(\sigma)>\dim(\cN^s)$.
		\item Each $p_n$ is $(\eta,T,\cN^u)$-Pliss for $(\Psi_t^{*,X_n})_{t\in\RR}$.
	\end{itemize}
	Then for $n$ large enough, $W^u({\rm Orb}(p_n))\pitchfork W^s(\sigma^{X_n})\neq\emptyset$, where $\sigma^{X_n}$ is the continuation of $\sigma$ for $X_n.$
\end{Theorem}

\section{Compactification of flows}
We define now the blowup of a manifold associated to a $C^1$ vector field $X\in \cX^1(M)$.

\subsection{Blow-up of  manifolds and of invariant sets}~\label{s.blowup-of-manifold}
Consider the Grassmannian manifold 
$$\cG^1(M):=\{L: \textrm{ $L$ is a 1-dimensional linear space in $T_xM$, for $x\in M$}\}$$
and let $\widetilde \theta: \cG^1(M)\to M$ be the natural  projection. Given a $C^1$ vector field $X$, the flow $(\varphi_t)_{t\in\RR}$ generated by $X$ induces a flow $(\widetilde\varphi_t)_{t\in\RR}$ on $\cG^1(M)$ defined by 
$\widetilde\varphi_t(L)=D\varphi_t(L).$ 
One introduces
$$\widehat M={\rm Closure}\big(\big\{\RR\cdot X(x):~x\in M\setminus{\rm Sing}(X)\big\}\big).$$
{
Following~\cite{BdL} (see also a related construction in~\cite{LGW}),
for any invariant compact set $\Lambda$ of $(\varphi_t)_{t\in\RR}$, one associates
a subset  $\widehat \Lambda\subset \widehat M$ as follows.

At each singularity $\sigma\in \Lambda$, one considers the Oseledets splitting
$T_\sigma M= E_1(\sigma)\oplus\dots\oplus E_k(\sigma)$ and the corresponding Lyapunov exponents
$\lambda_1<\dots<\lambda_k$. For each $1\leq i\leq k$ such that $\lambda_i<0$,
there exists an invariant (strong) stable manifold $W^{s}_i(\sigma)$ that is tangent to  the stable space $E^s_i(\sigma):=E_1(\sigma)\oplus\dots\oplus E_i(\sigma)$. Similarly if $\lambda_j>0$,
there exists an invariant (strong) unstable manifold $W^{u}_j(\sigma)$ that is tangent to
the unstable space $E^u_j(\sigma):=E_j(\sigma)\oplus\dots\oplus E_k(\sigma)$.

The \emph{escaping stable space} is the biggest
stable space $E^s_i(\sigma)$ such that $W^s_i(\sigma)\cap \Lambda=\{\sigma\}$.
The \emph{escaping unstable space} is the biggest
unstable space $E^u_j(\sigma)$ such that $W^u_j(\sigma)\cap \Lambda=\{\sigma\}$.
The \emph{active center space} of $\sigma$ is then the sum
$E^c(\sigma):=E_{i+1}(\sigma)\oplus \dots\oplus E_{j-1}(\sigma)$ of Oseledets subspaces that are not involved in the
escaping stable and unstable spaces. Hence
\begin{equation}\label{e.escaping-splitting}
T_\sigma M= E^s_i(\sigma)\oplus E^c(\sigma)\oplus E^u_j(\sigma).
\end{equation}

We then define:
$$ \widehat \Lambda=\big\{\RR\cdot X(x):~x\in \Lambda\setminus{\rm Sing}(X)\big\}\cup \bigcup_{\sigma\in \sing(X)\cap \Lambda} \big\{\RR\cdot v:~v\in E^c(\sigma)\setminus \{0\}\big\}.$$
Note that $\widehat M$ and $\widehat \Lambda$ are invariant under the restriction
$(\widehat\varphi_t)_{t\in\RR}$ of the flow $(\widetilde\varphi_t)_{t\in\RR}$.
We call $\widehat M$ (resp. $\widehat \Lambda$) the {\it blow-up} of $M$ (resp. $\Lambda$). 
We denote by $\theta: \widehat M\rightarrow M$ the restriction of the projection $\widetilde \theta: \cG^1(M)\to M$.
}

\begin{Remark}
{\rm (1)} For $x\in M\setminus \sing(X)$, there is a unique $\widehat x\in \widehat M$
satisfying $\theta(\widehat x)=x$: this is $\widehat x=\RR\cdot X(x)$.
And $\theta$ is a homeomorphism between neighborhoods of $\widehat x\in \widehat M$
and of $x\in M$.
\smallskip

\noindent
{\rm (2)} $\theta(\widehat M)=M$ when $\Sing(X)$ has empty interior
(e.g. when $X$ is multi-singular hyperbolic).

{
\noindent
{\rm (3)} \cite[Proposition 37]{BdL} shows that
when $\Lambda$ contains only finitely many singularities,
the set $\widehat \Lambda$ is compact
(it is in fact true for general sets $\Lambda$, but we will not need it).}
\end{Remark}

{
In the following we will discuss and use two important properties of this blow-up.
One of them states that if the linear Poincar\'e flow on $\Lambda$ admits a dominated splitting, then
this splitting extends to the blow-up $\widehat \Lambda$ under some conditions
(see Lemma~\ref{Lem:dominated-manifold-extended}).
The other one is the following (a similar proof has  appeared in \cite[Proposition 41]{BdL} and \cite[Lemma 4.4]{LGW}):

\begin{Lemma}\label{l.lift-omega-limit}
Let $\wh\Lambda\subset \wh M$ be the blow-up of an invariant compact set $\Lambda$ whose singularities are hyperbolic.
For any $x\in M\setminus W^s(\sing(X))$ such that $\omega(x)\subset \Lambda$, the
lift $\wh x$ satisfies $\omega(\wh x)\subset \wh \Lambda$.
\end{Lemma}
\begin{proof}
Since $\theta(\omega(\widehat x))=\omega(x)$,
and $\theta^{-1}(z)$ is a singleton for non-singular points $z\in M$,
any point $\widehat z\in \omega(\widehat x)$ whose projection $\theta(\widehat z)$ is not singular
belongs to $\wh \Lambda$. We thus need to consider points $\widehat z$ in $\omega(\widehat x)$
whose projection is a singularity $\sigma$. By definition $\widehat z$ is a
line $L\in\cG^1(M)$.

We claim that $L\subset E^c(\sigma)\oplus E_j^u(\sigma)$. 
By definition, there is a sequence of times $t_n\to +\infty$ such that $\lim_{n\to+\infty}\varphi_{t_n}(x)=\sigma$ in $M$ and $\lim_{n\to+\infty}\RR\cdot X(\varphi_{t_n}(x))=L\in\cG^1(M)$. 
We choose a small compact neighborhood $U$ of $\sigma$.
Since $x$ does not belong to $W^s(\sigma)$,
for each $t_n$ there is $s_n>0$ such that $\varphi_{s_n}(x_n)$ belongs to the boundary of $U$, and $\varphi_{[s_n,t_n]}(x)\subset U$. It is clear that $t_n-s_n\to+\infty$ and $s_n\to+\infty$ as $n\to+\infty$. Any limit point $y$  of the sequence $(\varphi_{s_n}(x))$ satisfies $y\in\omega(x)\subset\Lambda$. Since the positive orbit of $y$ is contained in $U$ and $\sigma$ is hyperbolic, the point $y$ is contained in the stable manifold of $\sigma$. By the assumption, $y\notin W^s_i(\sigma)$. Thus, for any large time $T$, the direction $\RR\cdot X(\varphi_T(y))$ is very close to $E^c(\sigma)$. Thus, for $n$ large enough, $\RR\cdot X(\varphi_{s_n+T}(x))$ is close to $E^c(\sigma)$, and
we have $t_n>s_n+T$. By the domination of the tangent flow near $\sigma$, the direction $\RR\cdot X(\varphi_{t_n}(x)$ is close to $E^c(\sigma)\oplus E_j^u(\sigma)$. This implies the announced claim on $L$.

A symmetric argument gives $L\subset E_i^s(\sigma)\oplus E^c(\sigma)$. Hence $L\subset E^c(\sigma)$ and
$\widehat z=L$ belongs to $\wh \Lambda$.
\end{proof}
}

\subsection{Lifts of measures}
Let us denote by $\cM(K)$ the set of  probability measures on a compact metric space $K$. Given a dynamical system on $K$, let $\cM_{\rm inv}(K)$ denote  the set of invariant probability measures on $K$.
We explore the relationship between invariant measures on $M$ and on $\widehat M$.

For an invariant measure $\mu$ supported on $M$, an invariant measure $\widehat\mu$ supported on $\widehat M$ is said to be a \emph{lift} of $\mu$ if $\theta_*(\widehat\mu)=\mu$. The measure $\mu$ is called \emph{projection} of $\widehat\mu$. Since $\theta$ is continuous and since $\theta(\widehat M)=M$ when $\Sing(X)$ has empty interior, the following properties hold:
\begin{Lemma}\label{Lem:toM}
The projection of any ergodic measure $\widehat\mu$ of $(\widehat\varphi_t)_{t\in\RR}$
is ergodic. 
If all singularities are hyperbolic and $\Lambda$ is an invariant compact set,
any  $\mu\in\cM_{\rm inv}(\Lambda)$ has an invariant lift on $\widehat \Lambda$.	
	
\end{Lemma}
Now, we present the relationship between the metric entropy of an invariant measure and of its lift. 
The following theorem follows from \cite[Theorem 1]{SV}, see also \cite{SS}.
\begin{Theorem}\label{Thm:same-entropy}
Let $X$ be a $C^1$ vector field whose 
singularities are all hyperbolic.
Then any invariant measure $\widehat\mu$ of $(\widehat\varphi_t)_{t\in\RR}$ and its projection $\mu$ satisfy
${\rm h}_{\widehat\mu}(\widehat\varphi_1)={\rm h}_{\mu}(\varphi_1)$.

As a consequence, if  $\supp(\mu)\subset {\rm Sing}(X)$, then  ${\rm h}_{\widehat\mu}(\widehat\varphi_1)=0$.
\end{Theorem}

\subsection{The extended linear Poincar\'e flow}\label{Sec:compactification}
\paragraph{The pull back bundle $\widetilde {TM}$ and flow $\widetilde \Phi$.}
Let us consider the pull-back bundle $\widetilde{TM}\to \cG^1(M)$
of the tangent bundle $TM\to M$ through the projection $\widetilde \theta\colon \cG^1\to M$.
It is endowed with a linear flow $(\widetilde \Phi_t)_{t\in \RR}$ over the flow $(\widetilde \varphi_{t})_{t\in \RR}$
which is simply the pull back of the tangent flow $(D\varphi_t)_{t\in \RR}$.
One gets a projection $\widetilde \Theta\colon \widetilde{TM}\to TM$
which is an isometry on each fiber.

\paragraph{The normal bundle $\widetilde {N}$ and flow $\widetilde \Psi$.}
We then introduce the normal bundle $\widetilde \cN\to \cG^1(M)$:
any $L\in \cG^1(M)$ is a one-dimensional  line $L\subset T_xM$, and
the fiber $\widetilde \cN_L$ is defined as
$$\widetilde \cN_L:=\{u\in \widetilde{TM}_L: \widetilde \Theta(u)\perp L\}.$$
Let us introduce the orthogonal projection $\widetilde p\colon \widetilde {TM}\to \widetilde \cN$:
for any $L\in \cG^1(M)$ and $u\in \widetilde \cN_L$,
$$
\widetilde p(u):=u-\frac{<u, v>}{\|v\|^2}\cdot v,$$
where $v$ is any non-zero vector in $\widetilde{TM}$
such that $\widetilde \Theta(v)\in L$ (the definition of $\widetilde p(u)$ does not depend on this choice).
One defines a linear flow $(\widetilde \Psi_t)_{t\in \RR}$ on $\widetilde \cN$ (see~\cite[page 195]{L})
over the flow $(\widetilde \varphi_t)_{t\in \RR}$:
$$
\widetilde \Psi_t(u):=\widetilde p\circ \widetilde \Phi_t(u).$$
These constructions are summarized as follows:
\begin{displaymath}
	\xymatrix{   (\widetilde \cN, \widetilde \Psi)\ar[rd]  &(\widetilde {TM},\widetilde \Phi) \ar[l]^{\widetilde p} \ar[r]^{\widetilde \Theta} \ar[d] & (TM,D\varphi) \ar[d]\\
		&(\cG^1(M),\widetilde \varphi)  \ar[r]^{\widetilde \theta}& (M,\varphi)}
\end{displaymath}

\paragraph{The extended bundle $\widehat {\cN}$ and flow $\widehat \Psi$.}
Let $\widehat \cN_{\widehat M}$ be the restriction of the bundle $\widetilde \cN$ to the subset $\widehat M\subset \cG^1(M)$
and let $(\widehat \Psi_t)_{t\in \RR}$ and $\Theta$ be the restriction to $\widehat \cN$ of the flow $(\widetilde \Psi_t)_{t\in \RR}$
and of the projection $\widetilde \Theta$.
Note that $\Theta$ identifies $\widehat \cN_{\RR\cdot X(x)}=\widehat \cN_{\theta^{-1}(x)}$ with $\cN_x$ for any $x\in M\setminus \sing(X)$,
and through this identification the flow $(\Psi_t)_{t\in \RR}$ and $(\widehat \Psi_t)_{t\in \RR}$ coincide.
Moreover
$\widehat \cN$ is the closure of $\Theta^*\cN$, hence $(\widehat \cN,\widehat \Psi)$
can be thought as a compactification of $(\cN,\Psi)$.
When one considers an invariant compact set $\Lambda\subset M$, the normal bundle $\widehat\cN_{\widehat \Lambda}$
with an extended flow $\widehat \Psi$ can be defined similarly.

The flow $(\widehat \Psi_t)_{t\in \RR}$, called \emph{extended linear Poincar\'e flow},
was defined by~Li-Gan-Wen \cite{LGW}.
\begin{displaymath}
	\xymatrix{   (\widehat {\cN},\widehat \Psi) \ar[d]&\\
		(\widehat M,\widehat \varphi)  \ar[r]^{\theta}& (M,\varphi)}
\end{displaymath}

\subsection{Dominated splittings of the extended linear Poincar\'e flow}
Let $\Lambda$ be a $(\varphi_t)_{t\in\RR}$ invariant compact  set and $\widehat \Lambda\subset \widehat M$
its blow-up. A $(\widehat\Psi_t)_{t\in\RR}$-invariant splitting $\widehat\cN_{\widehat\Lambda}=\widehat\cE\oplus\widehat\cF$  is \emph{dominated} if there are $\eta,T>0$ such that
$$\|\widehat\Psi_T|_{\widehat\cE({L})}\|\cdot \|\widehat\Psi_{-T}|_{\widehat\cF(\widehat\varphi_T({L}))}\|\leq  e^{-\eta},~~~\textrm{ for any ${L}\in\widehat\La$}.$$ 
	
{
\begin{Lemma}\label{Lem:dominated-manifold-extended}
If $\Lambda\subset M$ has a dominated splitting $\cN_{\Lambda\setminus{\rm Sing}(X)}=\cE\oplus\cF$ for $(\Psi_t)_{t\in\RR}$, if the singularities in $\Lambda$ are hyperbolic and if any singularity $\sigma\in \sing(X)\cap \Lambda$
satisfies $\dim(E^s_i(\sigma))\geq \dim(\cE)$ or $\dim(E^u_j(\sigma))\geq \dim(\cF)$, where
$T_\sigma M= E^s_i(\sigma)\oplus E^c(\sigma)\oplus E^u_j(\sigma)$ is the splitting
introduced in~\eqref{e.escaping-splitting},
then the blow-up $\widehat\Lambda\subset \widehat M$ admits a dominated splitting $\widehat\cN_{\widehat\Lambda}=\widehat\cE\oplus\widehat\cF$ for $(\widehat\Psi_t)_{t\in\RR}$ such that
\begin{equation}\label{e.domination-lift}
\Theta(\widehat\cE({\widehat x}))=\cE(x) \;\text{ and }\; \Theta(\widehat\cF({\widehat x}))=\cF(x)\;\;
\text{ for any $x\in\Lambda\setminus{\rm Sing}(X)$}.
\end{equation}
\end{Lemma}
\begin{Remark}\label{r.splitting-muti-singular}
If $\Lambda$ is multi-singular hyperbolic and if $\cN_{\Lambda\setminus \sing(X)}=\cN^{s}\oplus \cN^{u}$
is the associated dominated splitting, then for any singularity $\sigma\in \Lambda$,
either $E^{ss}(\sigma)$ or $E^{uu}(\sigma)$ is escaping, and
$\dim(E^{ss}(\sigma))=\dim(\cN^{s})$, $\dim(E^{uu}(\sigma))=\dim(\cN^{u})$.
So the Lemma can be applied and gives a dominated splitting
$\wh\cN_{\wh\Lambda}=\wh\cN^{s}\oplus \wh\cN^{u}$.
\end{Remark}
\begin{proof}[Proof of Lemma~\ref{Lem:dominated-manifold-extended}]
Since $\Theta$ is an isometry between fibers of
$\wh \cN_{\wh M\setminus \theta^{-1}(\sing(X))}$ and $\cN_{M\setminus \sing(X))}$,
the dominated splitting $\cN_{\Lambda\setminus{\rm Sing}(X)}=\cE\oplus\cF$
induces a dominated splitting
$\wh \cN_{\wh M\setminus \theta^{-1}(\sing(X))}=\wh \cE\oplus \wh \cF$
satisfying~\eqref{e.domination-lift}. This dominated splitting extends to the closure
of $\big\{\RR\cdot X(x):~x\in \Lambda\setminus{\rm Sing}(X)\big\}$, so by definition of
$\widehat \Lambda$, it remains to show that there exists a dominated splitting
over each compact set $\big\{\RR\cdot v:~v\in E^c(\sigma)\setminus \{0\}\big\}\subset \wh M$ for each $\sigma\in
\Lambda\cap \sing(X)$. We now fix such a singularity $\sigma$.
We can assume $\dim(E^c(\sigma))>0$ otherwise the set $\big\{\RR\cdot v:~v\in E^c(\sigma)\setminus \{0\}\big\}$
is empty and there is nothing to prove.

We will consider the case $\dim E_j^u(\sigma)\ge \dim (\cF)$;
the other case $\dim E_i^s(\sigma)\ge \dim (\cE)$ can be treated similarly.
Since $E^c(\sigma)$ is non-trivial and $\sigma$ is hyperbolic, $\sigma$ is accumulated by points of $\Lambda$
along some direction in $E^c(\sigma)$.
Then, the argument in \cite[Lemma 5.1]{GY} (using $\dim E_j^u(\sigma)\ge \dim (\cF)$)
implies that the Oseledets splitting at $\sigma$ satisfies
$E_j^u(\sigma)=E^{wu}(\sigma)\oplus E^{uu}(\sigma)$,
where $\dim E^{uu}(\sigma)=\dim(\cF)$ and the moduli of the eigenvalues associated to
$E^{wu}(\sigma)$ are smaller than the moduli of the eigenvalues associated to
$E^{uu}(\sigma)$. Possibly $\dim(E^{wu}(\sigma))=0$.

Now we consider any line $L\subset E^c(\sigma)$. We define $\widehat\cF(L)$ to be the orthogonal space of $L$ in $L\oplus E^{uu}(\sigma)$.
Since $E^{uu}(\sigma)$ and $E^c$ are invariant, if $L'=\wh \Psi_t(L)$ then
$\widehat\cF(L')=\wh \Psi_t(\wh \cF(L))$. We define in a similar way
$\widehat\cE(L)$ to be the orthogonal space of $L$ in {$E^{ss}_i(\sigma)\oplus E^c(\sigma)\oplus E^{wu}(\sigma)$.} 
	We also have
$\widehat\cE(L')=\wh \Psi_t(\wh \cE(L))$.

Now we check the domination. We have the following estimates:
\begin{itemize}
\item Let us write $E^{cs}(\sigma):=E^{ss}_i(\sigma)\oplus E^c(\sigma)\oplus E^{wu}(\sigma)$.
Since $\wh\cE(L)$ is contained in the invariant space $E^{cs}(\sigma)$,
the map $\widehat\Psi_t|_{\widehat\cE(L)}$ is obtained from
{$D\varphi_t|_{E^{cs}(\sigma)}$} by an orthogonal projection and
$\|\widehat\Psi_t|_{\widehat\cE(L)}\|\leq \|D\varphi_t|_{E^{cs}(\sigma)}\|$
for any $t\in \RR$.
\item Since the angle between $L$ and $E^{uu}(\sigma)$
is bounded away from zero,
there exists $C>0$ such that
for any $v\in \wh \cF(L)$ and any $t\in \RR$,
we have
$\|\wh \Psi_t(v)\|\leq C \|D\varphi_t|_{E^{uu}(\sigma)}\|\cdot \|v\|$.
\end{itemize}
Combining them, one gets:
$$\|\widehat\Psi_t|_{\widehat\cE(L)}\|\cdot {\|\widehat\Psi_{-t}|_{\widehat\cF(\wh \varphi_t(L))}\|}\le C\cdot {\|D\varphi_t|_{E^{cs}(\sigma)}\|}\cdot{\|D\varphi_{-t}|_{E^{uu}(\sigma)}\|}\le 
C{\rm e}^{-\eta t},$$
for some $\eta>0$ and any $t>0$ large.
This proves that the decomposition
$\wh \cN=\wh\cE\oplus \wh\cF$ is dominated above the set
$\theta^{-1}(\sigma)\cap \widehat \Lambda$.
\end{proof}
}

The sub-bundles $\widehat\cE,\widehat\cF$ of the dominated splitting vary continuously on $\widehat \Lambda$.
The following property is standard. (The extensions are in general not $(\widehat \Psi_t)_{t\in \RR}$-invariant.)
\begin{Lemma}\label{Lem:bundle-extended}
 If $\widehat \Lambda\subset \wh M$ has a dominated splitting $\widehat\cN_{\widehat\Lambda}=\widehat\cE\oplus\widehat\cF$ for $(\widehat\Psi_t)_{t\in\RR}$, then the bundles $\widehat\cE,\widehat\cF$ can be extended continuously to a neighborhood of
$\widehat \Lambda$.
\end{Lemma}

One has the following integral presentation.
\begin{Lemma}\label{Lem:limit-integral} 
If $\widehat \Lambda\subset \wh M$ has a dominated splitting $\widehat\cN_{\widehat\Lambda}=\widehat\cE\oplus\widehat\cF$ for $(\widehat\Psi_t)_{t\in\RR}$, and if one fixes a continuous extension of $\widehat \cF$ then there exists a neighborhood $\widehat U$ of $\widehat \Lambda$ in $\widehat M$ such that for any point $L\in \widehat M$ whose forward orbit
is contained in $\widehat U$ and satisfies $\omega(L)\subset \widehat \Lambda$,
$$\lim_{t\to+\infty}\bigg(\frac{1}{t} \log\big|{\rm Det}\widehat\Psi_t|_{\widehat \cF({L})}\big|-\frac{1}{t}\int_{0}^{t}\log\big|{\rm Det}\widehat\Psi_1|_{\widehat \cF(\widehat\varphi_s({L}))}\big|{\rm d}s \bigg)=0.$$
\end{Lemma}
\begin{proof}
We fix some continuous extensions, that we still denote by $\wh\cE,\wh\cF$.
For $\alpha>0$, one defines the \emph{normal $cu$-cone of angle $\alpha$} 
at any point $L$ close to $\widehat \Lambda$:
	$$\wh\cC^{\cF}_\alpha(L)=\big\{\widehat v\in\widehat\cN_{L}:\widehat v=\widehat v^{\cE}+\widehat v^{\cF},\widehat v^{\cE}\in\wh\cE,\widehat v^{\cF}\in\wh\cF,\|\widehat v^{\cE}\|\le \alpha\cdot \|\widehat v^{\cF}\|\big\}.$$
The dominated splitting implies:

\begin{Claim}\label{claim:invariant-cone}
In the setting of Lemma~\ref{Lem:bundle-extended},
there exist  $\tau> 0$ and positive numbers $\alpha_n\underset{n\to\infty}\longrightarrow0$ such that
for any $n$ large, any $L\in\widehat M$, any $t>\tau$ the following holds:  if 
the $1/n$ neighborhood of $\widehat\Lambda$ contains $\widehat\varphi_{[0,t]}(L)$, then
$$\widehat\Psi_t(\wh\cC^{\cF}_{\alpha_n}(L))\subset \wh\cC^{\cF}_{\alpha_n}(\widehat\varphi_t(L)).$$
\end{Claim}

As ${\widehat\cN}_{\widehat\Lambda}=\widehat\cE\oplus\widehat\cF$ is dominated for $(\wh\Psi_t)_{t\in \RR}$, by the claim, up to shrinking $\wh U$, there exists $\tau^\prime>1$ such that
$\wh\Psi_t(\wh\cC_\alpha^{\cF}(L))\subset \wh \cC_\alpha^{\cF}(\wh\varphi_t(L))$
for any $L\in\wh U$ and $t\geq\tau^\prime$ with $\wh\varphi_s(L)\in\wh U$ for $s\in[0,t]$.		
		By the uniform continuity of $(\wh\Psi_s)_{s\in[0,1]}$:
		\begin{itemize}
			\item There exists $K>0$ such that $\sup_{s\in[0,1]}\sup_{L\in\wh M}|\log\|\wedge^{\dim(\wh\cF)}\widehat\Psi_s(L)\||<K,$ where $\wedge^k \widehat\Psi_s $ denotes the action induced by $\wh \Psi_s$
			on the space of $k$-dimensional linear spaces.
			\item For any $\e>0$, there exists a cone field $  \wh\cC_\beta^{\cF}\subset \wh\cC_\alpha^{\cF}$ such that for any $L\in\wh U$ and any $\dim(\wh\cF)$-dimensional linear space $F\subset \wh\cC_\beta^{\cF}(L)$,
			$$\big|\log|{\rm Det}\wh\Psi_s|_F|-\log|{\rm Det}\wh\Psi_s|_{\wh\cF(L)}|\big|<\varepsilon.$$
		\end{itemize}
		By domination, 
		there exist  a neighborhood $\wh V\subset\wh U$ of $\wh\Lambda$ and $\tau >\tau^\prime$ such that for any $L\in\wh V$ and $t\geq \tau $ with $\wh\varphi_s(L)\in\wh V$ for $s\in[0,t]$, we get $\wh\Psi_t(\wh\cC_\alpha^{\cF}(L))\subset \wh\cC^{\cF}_\beta(\wh\varphi_t(L)).$   
		
		Given $L\in\wh U$ with $\omega(L)\subset\wh\Lambda$, there exists $t_0>\tau^\prime$ such that $\wh\varphi_t(L)\in\wh V$ for any $t\geq t_0$. Therefore, for $t\geq t_0+\tau$, one has $\wh\Psi_t(\wh\cF(L))\subset \wh\cC^{\cF}_\beta(\wh\varphi_t(L)).$
		Given $t>t_0+\tau$ large enough, for any $s\in[0,1]$, by setting $k=k(t,s)=[t-s-t_0-\tau]$,
		\begin{align*}
			\log\big|{\rm Det}\widehat\Psi_t|_{\widehat\cF(L)}\big|&=\log\big|{\rm Det}\widehat\Psi_{t_0+\tau}|_{\widehat\cF(L)}\big|+\log\big|{\rm Det}\widehat\Psi_{s}|_{\wh\Psi_{t_0+\tau}(\widehat\cF(L))}\big|\\
			&\hspace{5mm}+\sum_{i=0}^{k-1}\log\big|{\rm Det}\widehat\Psi_1|_{\wh\Psi_{s+i+t_0+\tau}(\widehat\cF( L))}\big|+\log\big|{\rm Det}\widehat\Psi_{t-s-k-t_0-\tau}|_{\wh\Psi_{s+k+t_0+\tau}(\widehat\cF( L))}\big|.
		\end{align*}
		The choices of $t_0$ and $\tau$ give
		$  \wh\Psi_{s+i+t_0+\tau}(\widehat\cF( L))\subset \wh\cC^{\cF}_\beta(\wh\varphi_{s+i+t_0+\tau}(L)),$ for $i=0,\cdots,k$ and $s\in[0,1]$.
		By the choice of the cone field $\wh\cC^{\cF}_\beta$,
		\begin{align*}
			&\bigg|	\log\big|{\rm Det}\widehat\Psi_t|_{\widehat\cF(L)}\big|-\log\big|{\rm Det}\widehat\Psi_{t_0+\tau}|_{\widehat\cF(L)}\big|-\log\big|{\rm Det}\widehat\Psi_{s}|_{\widehat\cF(\wh\varphi_{t_0+\tau}(L))}\big|\\
			&\hspace{3mm}-\sum_{i=0}^{k-1}\log\big|{\rm Det}\widehat\Psi_1|_{\wh\cF(\wh\varphi_{s+i+t_0+\tau}( L))}\big|-\log\big|{\rm Det}\widehat\Psi_{t-s-k-t_0-\tau}|_{\widehat\cF(\wh\varphi_{s+k+t_0+\tau} (L))}\big|\bigg|<(k+2)\e.
		\end{align*}
		Thus, we get
		\begin{align*}
			\bigg|\frac{1}{t} \log&\big|{\rm Det}\widehat\Psi_t|_{\widehat \cF( {L})}\big| -\frac{1}{t}\int_{t_0+\tau}^{[t-t_0-\tau]-1+t_0+\tau}\log\big|{\rm Det}\widehat\Psi_1|_{\widehat\cF(\widehat\varphi_{s}(L))}\big|{\rm d}s\bigg|\\
			&=
			\bigg|\frac{1}{t} \log\big|{\rm Det}\widehat\Psi_t|_{\widehat \cF( {L})}\big| -\frac{1}{t}\int_{0}^{[t-t_0-\tau]-1}\log\big|{\rm Det}\widehat\Psi_1|_{\widehat\cF(\widehat\varphi_{s+t_0+\tau}(L))}\big|{\rm d}s\bigg|\\
			&=	
			\bigg|\frac{1}{t} \log\big|{\rm Det}\widehat\Psi_t|_{\widehat \cF( {L})}\big| -\int_0^1\frac{1}{t}\sum_{i=0}^{k-1}\log\big|{\rm Det}\widehat\Psi_1|_{\widehat\cF(\widehat\varphi_{s+i+t_0+\tau}(L))}\big|{\rm d}s\bigg|
			\\
			&=\bigg|\int_0^1\frac{1}{t}\log\big|{\rm Det}\widehat\Psi_t|_{\widehat\cF(L)}\big|{\rm d}s-\int_0^1\frac{1}{t}\sum_{i=0}^{k-1}\log\big|{\rm Det}\widehat\Psi_1|_{\widehat\cF(\widehat\varphi_{s+i+t_0+\tau}(L))}\big|{\rm d}s\bigg|
			\\
			&<\frac{([t_0+\tau]+2)K}{t}+\frac{(k+2)\e}{t}.
		\end{align*}
		Hence, this gives
		\begin{align*}
			\big|\frac{1}{t} \log&\big|{\rm Det}\widehat\Psi_t|_{\widehat \cF( {L})}\big| -\frac{1}{t}\int_{0}^{t}\log\big|{\rm Det}\widehat\Psi_1|_{\widehat\cF(\widehat\varphi_{s}(L))}\big|{\rm d}s\big|\\
			&\leq
			\big|\frac{1}{t} \log\big|{\rm Det}\widehat\Psi_t|_{\widehat \cF( {L})}\big| -\frac{1}{t}\int_{t_0+\tau}^{[t-t_0-\tau]-1+t_0+\tau}\log\big|{\rm Det}\widehat\Psi_1|_{\widehat\cF(\widehat\varphi_{s}(L))}\big|{\rm d}s\big|\\
			&\hspace{5mm}+ \frac{1}{t}\big|\int_{0}^{t_0+\tau}\log\big|{\rm Det}\widehat\Psi_1|_{\widehat\cF(\widehat\varphi_{s}(L))}\big|{\rm d}s+\int_{[t-t_0-\tau]-1+t_0+\tau}^{t}\log\big|{\rm Det}\widehat\Psi_1|_{\widehat\cF(\widehat\varphi_{s}(L))}\big|{\rm d}s\big|
			\\
			&<\frac{([t_0+\tau]+2)K}{t}+\frac{(k+2)\e}{t}+\frac{(t_0+\tau)K}{t}+\frac{K}{t}.
		\end{align*}
		Letting $t$ tend to infinity and by the arbitrariness of $\e$, one has 
		$$\lim_{t\rightarrow+\infty}\bigg(\frac{1}{t} \log\big|{\rm Det}\widehat\Psi_t|_{\widehat \cF( {L})}\big| -\frac{1}{t}\int_{0}^{t}\log\big|{\rm Det}\widehat\Psi_1|_{\widehat\cF(\widehat\varphi_{s}(L))}\big|{\rm d}s\bigg)=0.$$
The Lemma~\ref{Lem:limit-integral} is now proved.
\end{proof}

One can compare the integrated Jacobian along subbundles of $\cN$ and $\widehat \cN$.
	
\begin{Lemma}\label{Lem:extended-measure}
If $\Lambda\subset M$ has a dominated splitting $\cN_{\Lambda\setminus{\rm Sing}(X)}=\cE\oplus\cF$ for $(\Psi_t)_{t\in\RR}$ which lifts as a dominated splitting $\widehat\cN_{\widehat\Lambda}=\widehat\cE\oplus\widehat\cF$ for $(\widehat\Psi_t)_{t\in\RR}$,
if $\mu$ is an invariant measure supported on $\Lambda$ such that $\mu({\rm Sing}(X))=0$ and if $\widehat\mu$ is the lift of $\mu$, then
$$\int \log\big|{\rm Det}\widehat\Psi_1|_{\widehat \cF}\big|{\rm d}\widehat\mu=\int \log\big|{\rm Det}\Psi_1|_{\cF}\big|{\rm d}\mu.$$
\end{Lemma}
\begin{proof}
For any regular point $x$ and its lift $\widehat x$, the map
$\Theta\colon \widehat \cF_{\widehat x}\to \cF_x$ is an isometry. Hence:
$$\big|{\rm Det}\widehat\Psi_1|_{\widehat \cF(\widehat x)}\big|=\big|{\rm Det}\Psi_1|_{\cF(x)}\big|.$$
Let $\mu$ be an ergodic regular measure with lift $\widehat \mu$.
Let $x$ be a regular point whose forward orbit equidistributes towards $\mu$.
Its lift $\widehat x$ equidistributes towards $\widehat\mu$. Thus the equality holds for any ergodic regular measure.
One concludes with the ergodic decomposition theorem.
\end{proof}

\subsection{Dynamics at Lorenz like singularities}
We recall that the blow-up of a multi-singular hyperbolic set $\Lambda$
admits a dominated splitting for $(\wh \Psi_t)_{t\in \RR}$, see Lemma~\ref{Lem:dominated-manifold-extended} and Remark~\ref{r.splitting-muti-singular}.
\begin{Lemma}\label{l.splitting-singularity}
Let $\sigma$ be a singularity in a multi-singular hyperbolic set $\Lambda$ with the splitting $\cN_{\Lambda\setminus{\rm Sing}(X)}=\cN^s\oplus\cN^u$
and let $L\in \theta^{-1}(\sigma)\cap\widehat \Lambda$. Then
\begin{itemize}
\item either $L\oplus \widetilde \Theta(\widehat \cN^u_L)=E^c_\sigma\oplus E^{uu}_\sigma$,
\item or $L\subset E^{ss}_\sigma\oplus E^{c}_\sigma$ and $\widetilde \Theta(\widehat \cN^u_L)\cap (E^{ss}_\sigma\oplus E^{c}_\sigma)=\{0\}$.
\end{itemize}
\end{Lemma}
\begin{proof}
We first claim that $L\subset (E^{ss}_\sigma\oplus E^{c}_\sigma) \cup (E^{c}_\sigma\oplus E^{uu}_\sigma)$.
Indeed if this is not the case, then there exists a sequence $x_n\to \sigma$ in $\Lambda\setminus \sing(X)$
such that $\RR\cdot X(x_n)\to L$. This implies that $x_n$ accumulates $\sigma$ along a direction
not included in $(E^{ss}_\sigma\oplus E^{c}_\sigma) \cup (E^{c}_\sigma\oplus E^{uu}_\sigma)$.
Using the dominated splitting $T_\sigma M=E^{ss}_\sigma\oplus E^{c}_\sigma\oplus E^{uu}_\sigma$
one deduces that $x_n$ have iterates converging to some point in $W^{ss}(\sigma)$
and other iterates converging to some point in $W^{uu}(\sigma)$.
This contradicts the definition of singular domination on $\Lambda$
(Definition~\ref{d.singular-domination}).

In the case where $L\subset E^{c}_\sigma\oplus E^{uu}_\sigma$,
let $\cF^s_L$ be the orthogonal complement to $E^{c}_\sigma\oplus E^{uu}_\sigma$ inside $T_\sigma M$
and let $\cF^u_L$ be the orthogonal complement to $L$ inside $E^{c}_\sigma\oplus E^{uu}_\sigma$.
After identification by $\widehat \Theta$, we get a decomposition $\cF^s_L\oplus \cF^u_L$
of $\widehat \cN_L$ which is mapped by $\widehat \Psi_t$
to the corresponding decomposition of $\widehat \cN_{\widehat \varphi_t(L)}$ for each $t$ since
$E^{c}_\sigma\oplus E^{uu}_\sigma$ is invariant.
This shows that $\cF^s_L=\widehat \Theta(\cN^s_L)$ and $\cF^u_L=\widehat\Theta(\cN^u_L)$.
The first item of the lemma holds.

In the case where $L\subset E^{ss}_\sigma\oplus E^{c}_\sigma$, the same argument shows that
$\widehat \Theta(\cN^u_L)$ is the orthogonal complement of $E^{ss}_\sigma\oplus E^{c}_\sigma$,
hence the second item holds.
\end{proof}

\begin{Lemma}\label{Lem:project-expanding}
Let $\sigma$ be a Lorenz like singularity with its splitting $T_\sigma M=E^{ss}_\sigma\oplus E^c_\sigma\oplus E^{uu}_\sigma$,
let $L\subset T_\sigma M$ be a $1$-dimensional space and $F\subset \widetilde \cN_L$ be a subspace with dimension
$\dim(E^{uu})$.

If one of the following setting is satisfied:
\begin{itemize}
\item either $L\oplus \widetilde \Theta(F)=E^c_\sigma\oplus E^{uu}_\sigma$ and $\sigma$ does not have index $d-1$,
\item or $L\subset E^{ss}_\sigma\oplus E^{c}_\sigma$ and $\widetilde \Theta(F)\cap (E^{ss}_\sigma\oplus E^{c}_\sigma)=\{0\}$,
\end{itemize}
then $$\liminf_{t\to+\infty}\frac{1}{t}\log\big|{\rm Det}(\widetilde \Psi_t|_{F})\big|>0.$$
\end{Lemma}
\begin{proof}
In the first case,
since $L\oplus \widetilde \Theta(F)=E^c_\sigma\oplus E^{uu}_\sigma$, the definition of $\widetilde \Psi_t$ gives
$$\log\big|{\rm Det}(\widetilde \Psi_t|_{F})\big|=\log\big|{\rm Det}(D\varphi_t|_{E^c\oplus E^{uu}})\big|-\log\|D\varphi_t|_{L}\|.$$
Thus,
$\liminf_{t\to+\infty}\frac{1}{t}\log\big|{\rm Det}(\widetilde\Psi_t|_{F})\big|$
is larger than or equal to the sum of Lyapunov exponents of $D\varphi_t$ along $E^c_\sigma\oplus E^{uu}_\sigma$ minus the maximal Lyapunov exponent of $D\varphi_t$ along $E^c_\sigma\oplus E^{uu}_\sigma$.
When $\dim E^{uu}_\sigma\ge 2$, it is larger than $\lambda^c_\sigma+\lambda_\sigma^{uu}>0$,
as in the definition of a Lorenz like singularity. When $\dim(E^{uu})=1$, it is larger than $\lambda^c$,
which is positive since $\sigma$ does not have index $d-1$.
	
In the second case,
since $\widetilde \Theta(F)\cap (E^{ss}_\sigma\oplus E^c_\sigma)=\{0\}$,
the image $D\varphi_t(\widetilde \Theta(F))$ is  close to $E^{uu}_\sigma$ for $t>0$ large.
Since $D\varphi_t(L)\subset E^{ss}_\sigma\oplus E^c_\sigma$, its angle with $D\varphi_t(\widetilde \Theta(F))$
is bounded away from $0$. Then
$\liminf_{t\to+\infty}\frac{1}{t}\log\big|{\rm Det}(\widetilde \Psi_t|_{F})\big|$
is the sum of Lyapunov exponents of $D\varphi_t$ along $E^{uu}_\sigma$ and is positive.
\end{proof}

 Recall that a singularity $\sigma$ contained in an invariant compact set $\Lambda$ is \emph{active} in $\Lambda$ if $\sigma$ is hyperbolic and  both $W^s(\sigma)\cap(\Lambda\setminus\{\sigma\})$ and $W^u(\sigma)\cap(\Lambda\setminus\{\sigma\})$ are non-empty, see~\cite[Definition 1.2]{CYZ2}.
Arguing as in the proof of Lemma~\ref{l.splitting-singularity}, we get the following result (see also~\cite[Lemma 4.4]{LGW}).
\begin{Lemma}\label{l.limit-singularity}
Let $\Lambda$ be a multi-singular hyperbolic set for some $X\in \cX^1(M)$
containing an active singularity $\sigma$ such that $\lambda^c_\sigma>0$.
Let $p_n\to \sigma$ be regular periodic points of vector fields $X_n\to X$
such that ${\rm Orb}(p_n)$ converges to a subset of $\Lambda$ for the Hausdorff topology.
Then the accumulation set of the sequence of unit vectors $X_n(p_n)/\|X_n(p_n)\|$ is contained in
$E^{ss}_\sigma\oplus E^c_\sigma$.
\end{Lemma}
\begin{proof}
The accumulation set is invariant by iterations. If one assumes by contradiction that it is not contained in
$E^{ss}_\sigma\oplus E^c_\sigma$, then it intersects $E^{uu}_\sigma$.
Hence for $n$ large, there is an iterate $z_n=\varphi^{X_n}_{t_n}(p_n)$, close to $\sigma_{X_n}$, with a direction
$X_n(z_n)/\|X_n(z_n)\|$ close to $E^{uu}_{\sigma_{X_n}}$. The orbit of such a point escapes a neighborhood
of $\sigma_{X_n}$ close to $W^{uu}(\sigma_{X_n})$. Passing to the limit,
one deduces that there exists a point in $W^{uu}(\sigma)\setminus \{\sigma\}$
which is limit of points in the periodic orbits ${\rm Orb}(p_n)$.
Consequently, $\Lambda$ intersects $W^{uu}(\sigma)\setminus \{\sigma\}$.
The singular domination on $\Lambda$ implies that
$W^{ss}(\sigma)\cap \Lambda=\{\sigma\}$.
But since $\lambda^c_\sigma>0$,
the backward orbits of $p_n$ escape a neighborhood of $\sigma_{X_n}$ along $W^{ss}(\sigma_{X_n})$:
by passing to the limit this implies that
$W^{ss}(\sigma)\setminus \{\sigma\}$ intersects $\Lambda$, a contradiction.
\end{proof}

\section{SRB measures on multi-singular hyperbolic sets}
In this section, we prove Theorem~\ref{Thm:main-localized} and discuss some consequences on SRB measures. 

\subsection{Growth rate of volumes}
To a homeomorphism  $f$ of $M$ and $x\in M$,  one associates the set $\cM_x(f)$ of measures which
are the accumulation points of the sequence $\big\{\frac{1}{n}\sum_{i=0}^{n-1}\delta_{f^i(x)}\big\}_{n\in\NN}$. Given a continuous flow $\varphi=(\varphi_t)_{t\in\RR}$ one associates the set $\cM_x(\varphi)$ of accumulation points of 
the arc $\big(\frac{1}{t}\int_0^t\delta_{\varphi_s(x)}\ud s\big)_{t>0}$  when $t$ tends to infinity.  
\medskip

Using Yomdin theory, Burguet has obtained \cite{Bu} the following
result relating the growth rate of the volume and the entropy for $C^\infty$ diffeomorphisms.
	
\begin{Theorem}\label{Thm:burguet}
Let $f$ be a $C^\infty$ diffeomorphism on $M$. Then for Lebesgue almost every point $x\in M$, there is a measure $\mu\in\cM_x(f)$ such that for any $1\le k\le d$,
$${\rm h}_\mu(f)\ge\limsup_{n\rightarrow+\infty}\frac{1}{n}\log\|\wedge^kDf^n(x)\|.$$
\end{Theorem}

This immediately extends to vector fields, and it helps us  to exclude the negative Lyapunov exponents for the lift of $\delta_\sigma$ which comes from the flow direction.
\begin{Proposition}\label{Pro:flow-version-burguet}
Let  $\varphi=(\varphi_t)_{t\in\RR}$ be the flow defined by a $C^\infty$ vector field $X$ over $M$. Then for Lebesgue almost every point $x\in M$, there is a measure $\mu\in\cM_x(\varphi)$ such that for any $1\le k\le d$,
$${\rm h}_\mu(\varphi_1)\ge\limsup_{t\to+\infty}\frac{1}{t}\log\|\wedge^kD\varphi_t(x)\|.$$
\end{Proposition}
\begin{proof}
Let $\mu_1\in \cM_x(\varphi_1)$ given by Theorem~\ref{Thm:burguet}.
Then the measure $\mu=\int_0^1 (\varphi_s)_*\mu_1{\rm d}s$ belongs to $\cM_x(\varphi)$
and satisfies $h_\mu(\varphi_1)=h_{\mu_1}(\varphi_1)$.
The result then follows from
$$\limsup_{t\to+\infty}\frac{1}{t}\log\|\wedge^kD\varphi_t(x)\|=
\limsup_{n\in \NN}\frac{1}{n}\log\|\wedge^kD\varphi_n(x)\|.$$
\end{proof}

Since the linear Poincar\'e flow is the projection of the tangent flow to the normal bundle, one gets the following corollary.	
\begin{Corollary}\label{Cor:growth-linear-poincare}
Let $X\in \cX^\infty(M)$ defining the flows $\varphi=(\varphi_t)_{t\in\RR}$ and $(\Psi_t)_{t\in\RR}$.
Then for Lebesgue almost every $x\in M\setminus \Sing(X)$, there is a measure $\mu\in\cM_x(\varphi)$ such that for any $1\le k\le d-1$,
$${\rm h}_\mu(\varphi_1)\ge\limsup_{t\to+\infty}\frac{1}{t}\log\|\wedge^k\Psi_t(x)\|.$$
\end{Corollary}
	
The growth of the volume in the blow-up is related to invariant measures as follows.
\begin{Lemma}\label{Lem:burguet-Lemma-1}
Let us consider an invariant compact set $\wh \Lambda\subset \wh M$ for $(\wh \varphi_t)_{t\in \RR}$
which has a dominated splitting $\wh \cN_{\wh \Lambda}=\wh \cE\oplus\wh \cF$ for $(\wh \Psi_t)_{t\in\RR}$. Then there exists a neighborhood $\widehat U$ of $\widehat \Lambda$ in $\widehat M$ such that for any point $L\in \widehat M$ whose forward orbit
is contained in $\widehat U$ and which satisfies $\omega(L)\subset \widehat \Lambda$, one has
$$\limsup_{t\to+\infty}\frac{1}{t}\log\big|{\rm Det}\widehat\Psi_t|_{{\widehat \cF}(L)}\big|=\sup_{\widehat\nu\in\cM_L(\widehat\varphi)}\int \log\big|{\rm Det}\widehat\Psi_1|_{{\widehat \cF}}\big|{\rm d}\widehat\nu.$$
\end{Lemma}
\begin{proof}
By Lemma~\ref{Lem:limit-integral}, it suffices to prove
$$\limsup_{t\to+\infty}\frac{1}{t}\int_0^t\log\big|{\rm Det}\widehat\Psi_1|_{\widehat\cF(\widehat\varphi_s(L))}\big|{\rm d}s=\sup_{\widehat\nu\in\cM_{L}(\widehat\varphi)}\int \log\big|{\rm Det}\widehat\Psi_1|_{{\widehat \cF}}\big|{\rm d}\widehat\nu.$$
One first proves the inequality ``$\le$''. Let $T_i\to+\infty$ be a sequence realizing the limsup in the left hand side. Taking a subsequence if necessary, one gets a limit measure
$$\widehat\nu=\lim_{i\to+\infty}\frac{1}{T_i}\int_0^{T_i}\delta_{\widehat\varphi_s(L)}{\rm d}s.$$
One concludes by computing:
\begin{align*}
\int \log\big|{\rm Det}\widehat\Psi_1|_{{\widehat \cF}}\big|{\rm d}\widehat\nu&=\lim_{i\to\infty}\frac{1}{T_i}\int_0^{T_i}\log\big|{\rm Det}\widehat\Psi_1|_{\widehat\cF(\widehat\varphi_s({L}))}\big|{\rm d}s\\
&=\limsup_{t\to+\infty}\frac{1}{t}\int_0^t\log\big|{\rm Det}\widehat\Psi_1|_{\widehat\cF(\widehat\varphi_s({L}))}\big|{\rm d}s.
\end{align*}
\medskip
		
Now we prove the inequality``$\ge$''. For any $\varepsilon>0$, let $\widehat\mu\in\cM_{L}(\widehat\varphi)$ such that 
$$\int \log\big|{\rm Det}\widehat\Psi_1|_{{\widehat \cF}}\big|{\rm d}\widehat\mu>\sup_{\widehat\nu\in\cM_{L}(\widehat\varphi)}\int \log\big|{\rm Det}\widehat\Psi_1|_{{\widehat \cF}}\big|{\rm d}\widehat\nu-\varepsilon.$$
There exists a sequence $S_i\to+\infty$ satisfying
\begin{align*}\int \log\big|{\rm Det}\widehat\Psi_1|_{{\widehat \cF}}\big|{\rm d}\widehat\mu&=\lim_{i\to+\infty}\frac{1}{S_i}\int_0^{S_i}\log\big|{\rm Det}\widehat\Psi_1|_{\widehat\cF(\widehat\varphi_s({L}))}\big|{\rm d}s\\
&\leq \limsup_{t\to+\infty}\frac{1}{t}\int_0^t\log\big|{\rm Det}\widehat\Psi_1|_{\widehat\cF(\widehat\varphi_s({L}))}\big|{\rm d}s.
\end{align*}
Letting $\varepsilon\to 0$, one concludes.
\end{proof}

\subsection{Entropy formulas on the blow-up of a multi-singular hyperbolic set}
The following inequality holds for the metric entropy on the blow-up $\widehat \Lambda$
of a multi-singular hyperbolic set. We recall that there exists a dominated splitting
$\wh \cN_{\wh \Lambda}=\wh \cN^{s}\oplus\wh \cN^{u}$ (see Remark~\ref{r.splitting-muti-singular}).

\begin{Proposition}\label{Prop:area-expanded}
Let $X\in \cX^1(M)$, $\Lambda$ be a multi-singular hyperbolic set and $\widehat \mu$
be a $(\widehat \varphi_t)_{t\in \RR}$-invariant measure on $\widehat\Lambda$
such that $\widehat \mu(\theta^{-1}(\sigma))=0$ for any singularity $\sigma$ with index $d-1$. Then,
$${\rm h}_{\widehat\mu}(\widehat\varphi_1)\le \int \log\big|{\rm Det}\widehat\Psi_1|_{{\widehat \cN}^{u}}\big|{\rm d}\widehat\mu.$$
The equality holds if and only if the projection $\mu=\theta_*(\widehat\mu)$ is an SRB measure.  
\end{Proposition}
\begin{proof} We first prove a preliminary property.
\begin{Claim}
$\displaystyle\int \log\big|{\rm Det}\widehat\Psi_1|_{{\widehat \cN}^{u}}\big|{\rm d}\widehat\mu>0$.
\end{Claim}
\begin{proof}
By the ergodic decomposition theorem, one can assume that $\widehat\mu$ is ergodic.
By Lemma~\ref{Lem:toM} its projection $\mu=\theta_*(\widehat \mu)$ is ergodic.
When $\mu$ is regular, Lemma~\ref{Lem:extended-measure} and Proposition~\ref{Prop:sum-positive-jacobian} give
$$\int \log\big|{\rm Det}\widehat\Psi_1|_{{\widehat \cN}^{u}}\big|{\rm d}\widehat\mu=\int \log\big|{\rm Det}\Psi_1|_{{\cN}^{u}}\big|{\rm d}\mu=\sum\lambda^+(\mu)>0.$$
		
Now we consider the case where $\mu$ is the Dirac measure at a singularity $\sigma$.
Let $L\in {\rm Supp}(\widehat \mu)$ be a typical point of $\widehat\mu$.
By Lemma~\ref{l.splitting-singularity}, and since $\sigma$ does not have index $d-1$,
the subspace $F:=\widehat \cN^u_L$ satisfies the assumptions of Lemma~\ref{Lem:project-expanding}.
Hence,
$$\liminf_{t\to+\infty}\frac{1}{t}\log\big|{\rm Det}(\widehat \Psi_t|_{\widehat {\cN^u_L}})\big|>0.$$
Lemma~\ref{Lem:limit-integral} gives $\displaystyle\frac{1}{t}\int_{0}^{t}\log\big|{\rm Det}\widehat\Psi_1|_{{\widehat \cN}^{u}(\widehat\varphi_s({L}))}\big|{\rm d}s>0$ and
the ergodic theorem concludes.
\end{proof}
	
For almost all ergodic components $\widehat\nu$ of $\widehat\mu$, the projection $\nu$ is ergodic by Lemma~\ref{Lem:toM}.
If $\nu$ is regular, then $\nu$ is hyperbolic by Proposition~\ref{Prop:sum-positive-jacobian}. 
Then, 
\begin{align*}
			{\rm h}_{\widehat\nu}(\varphi_1)
			&={\rm h}_{\nu}(\varphi_1) 
			&\text{\scriptsize by Theorem~\ref{Thm:same-entropy},} \\
			&\le \sum\lambda^+(\nu) 
			&\text{\scriptsize by Ruelle's inequality (Theorem~\ref{Thm:Ruelle-inequality}),} \\
			&=\int \log\big|{\rm Det}\Psi_1|_{{\cN}^{u}}\big|{\rm d}\nu &\text{\scriptsize by Proposition~\ref{Prop:sum-positive-jacobian},} \\
			&=\int \log\big|{\rm Det}\widehat\Psi_1|_{{\widehat \cN}^{u}}\big|{\rm d}\widehat\nu
			&\text{\scriptsize by Lemma~\ref{Lem:extended-measure}.}
\end{align*}
If $\nu$ is singular, by the claim and Theorem~\ref{Thm:same-entropy}
the inequality holds also (and is strict in this case). Thus, the inequality holds for $\widehat\mu$ by the ergodic decomposition theorem.
		
Recall that any   ergodic measure on $\Lambda$ has positive Lyapunov exponents.  By the arguments above, the equality holds if and only if for almost any ergodic component $\widehat\nu$ of $\wh\mu$, the projection $\nu$ is regular
and satisfies ${\rm h}_{\nu}(\varphi_1) =\sum\lambda^+(\nu)$, i.e. is an SRB measure.
Hence the equality holds if and only if  the projection $\mu$ of $\widehat\mu$ is an SRB measure.  
\end{proof}

For limit measures, one allows singularities of index $d-1$ and we obtain:

\begin{Proposition}\label{e.empirical-singularity}
Let $X\in \cX^1(M)$, let $\Lambda$ be a multi-singular hyperbolic set with no isolated singularities,
let $x\in \Basin(\Lambda)$ be a regular point and
$\widehat \mu$ be a limit measure of
$\big(\frac{1}{t}\int_0^t\delta_{\widehat \varphi_s(\wh x)}\ud s\big)_{t>0}$
on $\wh \Lambda$.
If $\mu:=\theta_*(\widehat \mu)$ gives positive mass to a singularity of index $d-1$,
then
$${\rm h}_{\widehat\mu}(\widehat\varphi_1)< \int \log\big|{\rm Det}\widehat\Psi_1|_{{\widehat \cN}^{u}}\big|{\rm d}\widehat\mu.$$
\end{Proposition}
\begin{proof}
Consider a singularity $\sigma$ of index $d-1$ with
$\mu(\{\sigma\})>0$,
and decompose $\wh\mu$ as $\wh \mu_1 +\wh\mu_2$ where
$\theta_*(\wh\mu_1)$ is supported on $\sigma$ and $\theta_*(\wh\mu_2)(\{\sigma\})=0$.
The normalisation of $\widehat\mu_1$ has zero entropy
(by Theorem~\ref{Thm:same-entropy}). By Proposition~\ref{Prop:area-expanded},
it suffices to prove
$\displaystyle\int \log\big|{\rm Det}\widehat\Psi_1|_{{\widehat \cN}^{u}}\big|{\rm d}\widehat\mu_1>0$.

We set $C=(\lambda^c_\sigma+\lambda^{uu}_\sigma)/2$ which is positive since $\sigma$
is a Lorenz like singularity. We also choose $\eta>0$ much smaller than $C$
and satisfying
\begin{equation}\label{e.choice-eta}
\mu(\{\sigma\})>\eta(1+4/C).
\end{equation}

By assumption $\sigma$ is Lorenz like with a splitting
$T_\sigma M=E_\sigma^{ss}\oplus E_\sigma^c\oplus E_\sigma^{uu}$
such that $\dim(E_\sigma^{uu})=\dim(E_\sigma^c)=1$ and $\lambda^c_\sigma<0$.
Since $\sigma$ is not isolated in $\Lambda$,
the strong manifold $W^{uu}(\sigma)$ intersects $\Lambda\setminus \{\sigma\}$.
The definition of singular domination then implies that
$W^{ss}(\sigma)\cap \Lambda=\{\sigma\}$.
In particular, if $y=\varphi_t(x)$
is close to $\sigma$, with $t>T$ for some $T>0$, then
$\RR\cdot X(y)$ is close to $E^{cu}_\sigma=E^c_\sigma\oplus E^{uu}_\sigma$.

Let us fix $\varepsilon>0$ small and a small closed neighborhood $V$ of $\sigma$ such that $\varphi_T(x)\not\in V$.
We can require $\theta_*\wh \mu(\partial V)=0$ and
by our choice of $\wh \mu_2$,
\begin{equation}\label{e.choiceV}
\int_{\theta^{-1}(V)} \bigg|\log\big|{\rm Det}\widehat\Psi_1|_{{\widehat \cN}^{u}(\wh x)}\big|\bigg|{\rm d}\wh \mu_2 <\eta.
\end{equation}

One considers the connected components of the set
$\{\varphi_t(x),t>T\}\cap V$.
They are defined by their entry and exit times $a\leq b$ in $V$ (possibly $b=+\infty$).
By \cite[Lemma 3.2]{PYY2} (see  also~\cite[Lemma 3.1]{SYY}),  the interval  $[a,b]$ contains two disjoint subintervals
$[a_1,b_1]$, $[a_2,b_2]$ with $a_2>b_1$, such that:
\begin{itemize}
\item their complement in $[a,b]$ has uniformly bounded length,
\item for $s\in [a_1,b_1]$,
$\RR\cdot X(\varphi_s(x))$ is close to $E^{c}_\sigma$,
and $\log\big|{\rm Det}\widehat\Psi_1|_{{\widehat \cN}^{u}(\wh \varphi_s(\wh x))}\big|$
is $\tfrac \varepsilon 2$-close to $\lambda^{uu}_\sigma$,
\item for $s\in [a_2,b_2]$,
$\RR\cdot X(\varphi_s(x))$ is close to $E^{uu}_\sigma$,
and $\log\big|{\rm Det}\widehat\Psi_1|_{{\widehat \cN}^{u}(\wh \varphi_s(\wh x))}\big|$
is $\tfrac \varepsilon 2$-close to $\lambda^{c}_\sigma$.
\end{itemize}

\begin{Claim}
There exists $t_0>0$
such that for any connected component $\{\varphi_t(x), t\in [a,b]\}$ of the set
$\{\varphi_t(x),t>T\}\cap V$,
and for any $t\in [a+t_0,b]$,
$$\int_a^t\log\big|{\rm Det}\widehat\Psi_1|_{{\widehat \cN}^{u}(\wh \varphi_s(\wh x))}\big| {\rm d}s > C\cdot(t-a).$$
\end{Claim}
\begin{proof}
For any connected component, the point $\varphi_a(x)$ is uniformly far from the singular set,
hence the quantity $\log(\|X(\varphi_t(x))\|/\|X(\varphi_a(x))\|)$ is bounded from above by some uniform constant $C_0>0$.

When $t\in [a_1,b_1]$, 
the quantity
$$\log(\|X(\varphi_t(x))\|/\|X(\varphi_{a_1}(x))\|)=\int_{a_1}^t
\tfrac{1}{\|X(\varphi_s(x))\|^2}
\left<X(\varphi_s(x)),DX(\varphi_s(x))\cdot X(\varphi_s(x))\right>{\rm d}s$$
is bounded from below by $(\lambda^c_\sigma-\tfrac \varepsilon 2)\cdot(t-a_1)$ because for $s\in[a_1,b_1]$, $DX(\varphi_s(x))$ is close to $DX(\sigma)$ and $\RR\cdot X(\varphi_s(x))$ is close to $E_\sigma^c$. 
Similarly when $t\in [a_2,b_2]$, we have
$$\log(\|X(\varphi_t(x))\|/\|X(\varphi_{a_2}(x))\|)\geq (\lambda^{uu}_\sigma-\tfrac \varepsilon 2)\cdot(t-a_2).$$

When $t\in [a+t_0,b]$ is large, we decompose $[a,t]$ as two subintervals of $[a_1,b_1]$
and $[a_2,b_2]$ and sets with bounded length. 
Note that $[a_1,b_1]$ is always non-empty.
And if $[a_2,b_2]$ is taken empty, we just let $a_2=b_2=t$.
Hence there exists a uniform $C_1>0$
such that
\begin{align*}
\log(\|X(\varphi_t(x))\|/\|X(\varphi_{a}(x))\|)&\geq
\lambda^c_\sigma(\min(b_1,t)-a_1)+\lambda^{uu}_\sigma(\max(a_2,t)-a_2)-\tfrac \varepsilon 2(t-a)-C_1\\
&\geq
\lambda^c_\sigma(b_1-a_1)+\lambda^{uu}_\sigma(t-a_2)-\tfrac \varepsilon 2(t-a)-C_1,
\end{align*}
since $\lambda^c_\sigma<0$ and $\lambda^{uu}_\sigma>0$.
We thus get
\begin{equation}\label{e.time-estimate}
C_0\geq \lambda^c_\sigma(b_1-a_1)+\lambda^{uu}_\sigma(t-a_2)-\tfrac \varepsilon 2(t-a)-C_1.
\end{equation}
Since $\lambda^c_\sigma<0$, this gives from~\eqref{e.time-estimate},
$$(|\lambda^c_\sigma|+\lambda^{uu}_\sigma)(b_1-a_1)\geq
\lambda^{uu}_\sigma(t-a)-
\lambda^{uu}_\sigma [(a_2-b_1)+(a_1-a)]
-\tfrac \varepsilon 2(t-a)-C_1-C_0.$$
By choosing $t_0$ large enough, $t-a$ is large. Hence we get 
\begin{equation}\label{e.ineq1}
 b_1-a_1\geq \tfrac{\lambda^{uu}_\sigma-\varepsilon}{|\lambda^c_\sigma|+\lambda^{uu}_\sigma}(t-a).
 \end{equation}
Similarly from~\eqref{e.time-estimate},
$$(|\lambda^c_\sigma|+\lambda^{uu}_\sigma)(t-a_2)
\leq
|\lambda^c_\sigma|(t-a)-
|\lambda^c_\sigma|[(a_1-a)+(a_2-b_1)]+
\tfrac \varepsilon 2(t-a)+C_0+C_1$$
and hence
\begin{equation}\label{e.ineq2}
t-a_2\leq \tfrac{|\lambda^{c}_\sigma|+\varepsilon}{|\lambda^c_\sigma|+\lambda^{uu}_\sigma}(t-a).
\end{equation}

Since $\log\big|{\rm Det}\widehat\Psi_1|_{{\widehat \cN}^{u}(\wh \varphi_s(\wh x))}\big|$
is $\tfrac \varepsilon 2$-close to $\lambda^{uu}_\sigma$ when $s\in[a_1,b_1]$ and
$\log\big|{\rm Det}\widehat\Psi_1|_{{\widehat \cN}^{u}(\wh \varphi_s(\wh x))}\big|$
is $\tfrac \varepsilon 2$-close to $\lambda^{c}_\sigma$ when $s\in[a_2,b_2]$ 
and using again that $[a,b]\setminus ([a_1,b_1]\cup[a_2,b_2])$ has bounded length,
there exists $C_2>0$ such that
when $t\geq a_2$ one gets with~\eqref{e.ineq1} and~\eqref{e.ineq2}:
\begin{align*}
\int_a^t\log\big|{\rm Det}\widehat\Psi_1|_{{\widehat \cN}^{u}(\wh \varphi_s(\wh x))}\big| {\rm d}s
&\geq \tfrac{\lambda^{uu}_\sigma-\varepsilon}{|\lambda^c_\sigma|+\lambda^{uu}_\sigma}(\lambda^{uu}_\sigma-\tfrac \varepsilon 2)(t-a)
-\tfrac{|\lambda^{c}_\sigma|+\varepsilon}{|\lambda^c_\sigma|+\lambda^{uu}_\sigma}(|\lambda^c_\sigma|+\tfrac \varepsilon 2)(t-a)-C_2\\
&\geq \big(\lambda^{c}_\sigma+\lambda^{uu}_\sigma-\tfrac 3 2\varepsilon
\big)(t-a)
-C_2.
\end{align*}
Having chosen $\varepsilon>0$ small and since $t-a$ is large, this
concludes since $C=(\lambda^c_\sigma+\lambda^{uu}_\sigma)/2>0$.
\end{proof}

Let $W\subset V$ be a small neighborhood of $\sigma$ such that
$$
\varphi_{[-t_0,t_0]}(\partial V)\cap W=\emptyset
\quad \text{and} \quad \theta_*\wh \mu(\partial W)=0.$$
As $\theta_*(\wh\mu_1)$ is supported on $\sigma$, thus 
\[\int \log\big|{\rm Det}\widehat\Psi_1|_{{\widehat \cN}^{u}}\big|{\rm d}\widehat\mu_1=\int_{\theta^{-1}(\{\sigma\})} \log\big|{\rm Det}\widehat\Psi_1|_{{\widehat \cN}^{u}}\big|{\rm d}\widehat\mu_1.\]
As a consequence, to show that $\displaystyle\int \log\big|{\rm Det}\widehat\Psi_1|_{{\widehat \cN}^{u}}\big|{\rm d}\widehat\mu_1>0$, it suffices to prove that $$\displaystyle\int_{\theta^{-1}(W)}\log\big|{\rm Det}\widehat\Psi_1|_{{\widehat \cN}^{u}(\wh x)}\big|{\rm d}\wh \mu_1>0.$$
Let us consider $t>0$ large such that
\begin{align}
\label{e.approximation-measure1-new}
\frac 1 t \int_{s\in [0,t], \varphi_s(x)\in W}1
{\rm d}s> \wh \mu(\{\sigma\})&-\eta,\\
\label{e.approximation-measure2-new}\bigg|\frac 1 t \int_{s\in [0,t], \varphi_s(x)\in V}
\log\big|{\rm Det}\widehat\Psi_1|_{{\widehat \cN}^{u}(\wh \varphi_s(\wh x))}\big|{\rm d}s
&-\int_{\theta^{-1}(V)} \log\big|{\rm Det}\widehat\Psi_1|_{{\widehat \cN}^{u}(\wh x)}\big|{\rm d}\wh \mu
\bigg| <\eta,\\
\label{e.approximation-measure2-norm}\bigg|\frac 1 t \int_{s\in [0,t], \varphi_s(x)\in V\setminus W}
\log\big|{\rm Det}\widehat\Psi_1|_{{\widehat \cN}^{u}(\wh \varphi_s(\wh x))}\big|{\rm d}s
&-\int_{\theta^{-1}(V\setminus W)} \log\big|{\rm Det}\widehat\Psi_1|_{{\widehat \cN}^{u}(\wh x)}\big|{\rm d}\wh \mu
\bigg| <\eta.
\end{align}
From~\eqref{e.choiceV} and~\eqref{e.approximation-measure2-new}, one has that
\begin{equation}\label{e.average-V}
\bigg|\frac 1 t \int_{s\in [0,t], \varphi_s(x)\in V}
\log\big|{\rm Det}\widehat\Psi_1|_{{\widehat \cN}^{u}(\wh \varphi_s(\wh x))}\big|{\rm d}s
-\int_{\theta^{-1}(V)}\log\big|{\rm Det}\widehat\Psi_1|_{{\widehat \cN}^{u}(\wh x)}\big|{\rm d}\wh \mu_1
\bigg| <2\eta.
\end{equation}
Let $\mathcal{J}_V$ be the set of maximal subintervals $[a,b]\subset [0,t]$
such that  $\varphi_{[a,b]}(x)\subset V$, and let $\mathcal{J}_W$ be the subset of $\mathcal{J}_V$ such that for any $[a,b]\in \mathcal{J}_W$, one has $\varphi_{[a,b]}(x)\cap W\neq\emptyset$. By the definition of $W$, one has that $b-a\ge t_0$ in this case.

By the definition of $\wh \mu_1$, one notices that $$\int_{\theta^{-1}(W)}\log\big|{\rm Det}\widehat\Psi_1|_{{\widehat \cN}^{u}(\wh x)}\big|{\rm d}\wh \mu_1=\int_{\theta^{-1}(V)}\log\big|{\rm Det}\widehat\Psi_1|_{{\widehat \cN}^{u}(\wh x)}\big|{\rm d}\wh \mu_1.$$
Thus, from the above equality and~\eqref{e.average-V}, one has that
\begin{align*}
\bigg|\frac 1 t \sum_{[a,b]\in {\mathcal J}_V}
\int_a^b
\log\big|{\rm Det}\widehat\Psi_1|_{{\widehat \cN}^{u}(\wh \varphi_s(\wh x))}\big|{\rm d}s
-\int_{\theta^{-1}(W)} \log\big|{\rm Det}\widehat\Psi_1|_{{\widehat \cN}^{u}(\wh x)}\big|{\rm d}\wh \mu_1
\bigg|\le 2\eta.
\end{align*}
By~\eqref{e.choiceV} and \eqref{e.approximation-measure2-norm}, one has that
\begin{align*}
&~~~~\bigg|\frac 1 t \sum_{[a,b]\in {\mathcal J}_V\setminus{\mathcal J}_W}\int_a^b
\log\big|{\rm Det}\widehat\Psi_1|_{{\widehat \cN}^{u}(\wh \varphi_s(\wh x))}\big|{\rm d}s\bigg|
\le \frac 1 t \sum_{[a,b]\in {\mathcal J}_V\setminus{\mathcal J}_W}\int_a^b
\bigg|\log\big|{\rm Det}\widehat\Psi_1|_{{\widehat \cN}^{u}(\wh \varphi_s(\wh x))}\big|\bigg|{\rm d}s\\
&\le \int_{\theta^{-1}(V\setminus W)}\bigg|\log\big|{\rm Det}\widehat\Psi_1|_{{\widehat \cN}^{u}(\wh x)}\big|\bigg|{\rm d}{\wh \mu}+\eta= \int_{\theta^{-1}(V\setminus W)}\bigg|\log\big|{\rm Det}\widehat\Psi_1|_{{\widehat \cN}^{u}(\wh x)}\big|\bigg|{\rm d}{\wh \mu_2}+\eta<2\eta.
\end{align*}
On the other hand, one has that by the claim and ~\eqref{e.approximation-measure1-new},
\begin{align*}
\frac 1 t \sum_{[a,b]\in {\mathcal J}_W}\int_a^b
\log\big|{\rm Det}\widehat\Psi_1|_{{\widehat \cN}^{u}(\wh \varphi_s(\wh x))}\big|{\rm d}s&\ge \frac 1 t \sum_{I\in {\mathcal J}_W}C.{\rm Length}(I)\\
&\ge C\cdot\frac 1 t \int_{s\in [0,t], \varphi_s(x)\in W}1
{\rm d}s>C\cdot( \wh \mu(\{\sigma\})-\eta).
\end{align*}
The estimates above and the choice of $\eta$ imply
$$\int_{\theta^{-1}(W)}\log\big|{\rm Det}\widehat\Psi_1|_{{\widehat \cN}^{u}(\wh x)}\big|{\rm d}\wh \mu_1>C\cdot( \wh \mu(\{\sigma\})-\eta)-4\eta>0.$$
\end{proof}

\subsection{Existence of SRB measures: proof of Theorem~\ref{Thm:main-localized}}
The following theorem from \cite[Appendix B]{CYZ}  gives the existence of physical measures on singular hyperbolic attractors. Some related works can be found in \cite{ALM,LeYa}.
\begin{Theorem}[Corollary B.1 and Theorem I in~\cite{CYZ}]~\label{Thm.pseudo-physical-measure-attractor}
Let $X\in\mathcal{X}^1(M)$ and $\Lambda$ be a singular hyperbolic set with the splitting $T_{\Lambda} M=E^{ss}\oplus E^{cu}$ such that ${\rm Leb}({\rm Basin}(\Lambda))>0$.
Then for Lebesgue almost every point $x\in \Basin(\Lambda)$, each limit measure $\mu\in\cM_x(\varphi)$  satisfies	\begin{equation}\label{e.entropy-formula}
{\rm h}_{\mu}(\varphi_1 )=\int\log|\Det D\varphi_1 |_{E^{cu}}|\ud\mu.
\end{equation}
If moreover $X\in\cX^{2}(M)$ and $\Lambda$ is an attractor, then $\Lambda$ supports a unique measure $\mu$ satisfying \eqref{e.entropy-formula}. In this case, $\mu$ is an ergodic  physical measure with $\Leb(\Basin(\mu))=\Leb(\Basin(\Lambda))$.
\end{Theorem}

Now we can prove Theorem~\ref{Thm:main-localized} which addresses the multi-singular hyperbolic case.
\begin{proof}[Proof of Theorem~\ref{Thm:main-localized}]
Let $X\in \cX^\infty(M)$ and $\Lambda$ be a multi-singular hyperbolic set that does not contain any sink and satisfies ${\rm Leb}(\rm Basin(\Lambda))>0$.
If $\Lambda$ is a hyperbolic singularity or a periodic orbit, then it is a sink contradicting our assumption that
$\Lambda$ does not contain any sink.
If $\Lambda$ admits an isolated singularity $\sigma$ such that $\Leb\big(\Basin(\{\sigma\})\big)>0$,
then $\sigma$ is a sink, contradicting again our assumptions.
If $\Lambda$ admits an isolated singularity $\sigma$ such that $\Leb\big(\Basin(\{\sigma\})\big)=0$,
then $\Lambda':=\Lambda\setminus \{\sigma\}$ still satisfies the assumptions of Theorem~\ref{Thm:main-localized}.
We can thus reduce to the case where there exists a multi-singular hyperbolic splitting $\cN_{\Lambda\setminus{\rm Sing}(X)}=\cN^{s}\oplus \cN^u$ with $\dim(\cN^u)>0$ and where $\Lambda$ has no isolated singularities.

Let $U$ be a small neighborhood of $\Lambda$. For Lebesgue almost every point $x\in U\cap {\rm Basin}(\Lambda)$,
one considers a measure $\mu\in\mathcal{M}_x(\varphi)$ given by Corollary~\ref{Cor:growth-linear-poincare}.
Then for $k=\dim\cN^u$,
$${\rm h}_\mu(\varphi_1)\ge\limsup_{t\to+\infty}\frac{1}{t}\log\|\wedge^k\Psi_t(x)\|.$$
		
We are going to work on the blow-up $\widehat \Lambda\subset \widehat M$.
By Lemma~\ref{Lem:dominated-manifold-extended} {and Remark~\ref{r.splitting-muti-singular}}, the lifts of bundles $\widehat\cN^s$ and $\widehat\cN^u$ form a dominated splitting for the extended linear Poincar\'e flow $(\widehat\Psi_t)_{t\in\RR}$. By Lemma~\ref{Lem:bundle-extended}, the bundle $\widehat\cN^u$ can be extended continuously to an open neighborhood $\widehat U\subset \widehat M$ of $\widehat\Lambda$.
The map $\theta$ is a homeomorphism between the open sets $\widehat M\setminus \theta^{-1}(\sing(X))$
and $M\setminus \sing(X)$. Hence the bundle $\widehat\cN^u$ induces
a continuous extension of $\cN^u$ to the open set $U':=\theta(\widehat U)\setminus \sing(X)$, that we still denote by $\cN^u$.
Note that $U'$ contains every regular point of $\Lambda$.
		
Since $\Leb\big(\bigcup_{\sigma\in\Sing(X)\cap\Lambda}W^s(\sigma)\big)=0$, for Lebesgue almost every $x\in {\rm Basin}(\Lambda)$, there is $\tau>0$ such that $\varphi_\tau(x)\in U'$. One considers the fiber $\cN^u(x):=\Psi_{-\tau}(\cN^u(\varphi_\tau(x)))$ and its lift $\widehat\cN^u({\widehat x})$.

From now on we fix the point $x$ and the measure $\mu\in \cM_x$. There exists $t_n\to +\infty$
such that
$\displaystyle\mu=\lim_{n\to\infty} \frac{1}{t_n}\int_0^{t_n}\delta_{\varphi_s(x)}\ud s$.
Let $\widehat \mu$ be a limit measure of the sequence $\big(\frac{1}{t_n}\int_0^{t_n}\delta_{\widehat \varphi_s(\widehat x)}\ud s\big)$.
{By Lemma~\ref{l.lift-omega-limit}, $\wh \mu$ is a lift of $\mu$ on $\widehat \Lambda$.} One then gets the following estimates.
\begin{align}
			{\rm h}_{\widehat\mu}(\widehat\varphi_1)
			&={\rm h}_{\mu}(\varphi_1)  
			&\text{\scriptsize (by Theorem~\ref{Thm:same-entropy})} \nonumber \\
			&\ge\limsup_{t\to\infty}\frac{1}{t}\log\big|{\rm Det}\Psi_t|_{{ \cN}^u({x})}\big| &\text{\scriptsize (by Corollary~\ref{Cor:growth-linear-poincare})} \nonumber\\
			&=\limsup_{t\to\infty}\frac{1}{t}\log\big|{\rm Det}\widehat\Psi_t|_{{\widehat \cN}^u({\widehat x})}\big|
			&\text{\scriptsize (each $\Theta\colon \widehat \cN^u(\widehat y)\to \cN^u(y)$ is an isometry)} \nonumber\\
			&=\limsup_{t\to\infty}\frac{1}{t}\log\big|{\rm Det}\widehat\Psi_t|_{{\widehat \cN}^u(\widehat\varphi_\tau({\widehat x}))}\big| & \nonumber\\
			&=\sup_{\widehat\nu\in\cM_{\widehat x}(\varphi)}\int \log\big|{\rm Det}\widehat\Psi_1|_{{\widehat \cN}^{u}}\big|{\rm d}\widehat\nu
			&\text{\scriptsize (by Lemma~\ref{Lem:burguet-Lemma-1})} \nonumber\\
			&\ge\int \log\big|{\rm Det}\widehat\Psi_1|_{{\widehat \cN}^{u}}\big|{\rm d}\widehat\mu.
\label{e.ineq-entropy}
\end{align}
In particular, the conclusion of Proposition~\ref{e.empirical-singularity} fails
and $\widehat \mu(\theta^{-1}(\sigma))=0$ for any singularity
$\sigma$ with index $d-1$. By Proposition~\ref{Prop:area-expanded},
$\displaystyle\int \log\big|{\rm Det}\widehat\Psi_1|_{{\widehat \cN}^{u}}\big|{\rm d}\widehat\mu
\ge {\rm h}_{\widehat\mu}(\widehat\varphi_1).$
With inequality~\eqref{e.ineq-entropy}, one gets
$\displaystyle{\rm h}_{\widehat\mu}(\widehat\varphi_1)=\int \log\big|{\rm Det}\widehat\Psi_1|_{{\widehat \cN}^{u}}\big|{\rm d}\widehat\mu$. 
Proposition~\ref{Prop:area-expanded} then implies that $\mu$ is an SRB measure. 
\end{proof}

\section{Physical measures for multi-singular hyperbolic flows}

In this section, we first discuss the uniqueness of physical measures on singular hyperbolic attractors.
We then give a criterion for an accumulation of chain recurrence classes to be an attractor.
At last, we prove the following statement which clearly implies Theorems~\ref{Thm:main-ergodic} and~\ref{Thm:main-generic-ergodic}.

\begin{Theorem}\label{Thm:main-restated}
There exists a $C^1$ open and dense subset $\cU$ of $\cX^*(M)$ and a dense G$_\delta$ subset
$\cG\subset \cU$ such that for any vector field $X\in \cU$, there exist finitely many attractors $\Lambda_1,\dots,\Lambda_k$. Each of them is either a sink or a singular hyperbolic homoclinic class.

If $X\in\cU$ is $C^\infty$ or $X\in \cG$, then it admits finitely many ergodic physical measures $\mu_1,\dots,\mu_k$, such that:
\begin{enumerate}
\item their supports coincide with $\Lambda_1,\dots,\Lambda_k$ respectively,
\item $\Leb\big(\bigcup_{i=1}^k\Basin(\mu_i)\big)=\Leb(M)$,
\item each $\mu_i$ which is not supported on a sink is an SRB measure.
\end{enumerate}
\end{Theorem}

\subsection{Physical measures for generic singular hyperbolic attractors}
We will show  the existence of physical measures on singular hyperbolic transitive attractors for $C^1$ generic vector fields, as a consequence of Theorem~\ref{Thm.pseudo-physical-measure-attractor} for $C^2$ vector fields.
Note that Qiu \cite{Q} got the result for uniformly hyperbolic attractors of $C^1$-generic diffeomorphisms.
\begin{Theorem}~\label{Thm:unique-physical-on-attractor}
	There exists a dense G$_\delta$ subset $\cG$ of $\cX^1(M)$ such that for any $X\in\cG$ and for any singular hyperbolic attractor $\Lambda$ of $X$,
	there exists a unique SRB measure $\mu$ on $\Lambda$. Moreover, $\mu$ is physical and $\Leb(\Basin(\mu))=\Leb(\Basin(\Lambda))$.
\end{Theorem}
To prove Theorem~\ref{Thm:unique-physical-on-attractor}, we need the following result.

\begin{Lemma}\label{Lem:equivalent-description}
Let $X\in\cX^1(M)$ and $\Lambda$ be a singular hyperbolic attractor of $X$ with the splitting $T_\Lambda M=E^{ss}\oplus E^{cu}$. 
For any invariant measure $\mu$ on $\Lambda$, the following properties are equivalent.
\begin{itemize}
\item  $\displaystyle{\rm h}_{\mu}(\varphi_1 )\ge\int\log|\Det(D\varphi_1 |_{E^{cu}})|\ud\mu,$
\item $\displaystyle{\rm h}_{\mu}(\varphi_1)= \int\log|\Det(D\varphi_1|_{E^{cu}})|\ud\mu,$
\item $\mu$ is an SRB measure.
\end{itemize}
\end{Lemma}
Note that SRB measures supported on a singular hyperbolic attractor are regular,
have only one zero Lyapunov exponent along $E^{cu}$ and the other Lyapunov exponents along $E^{cu}$ are positive.
\begin{proof}
	By Ruelle's inequality (Theorem~\ref{Thm:Ruelle-inequality}) and Proposition~\ref{Prop:sum-positive-jacobian}, each regular ergodic measure $\nu$ satisfies
	$${\rm h}_{\nu}(\varphi_1 )\le\sum\lambda^+(\nu)=\int\log|\Det(D\varphi_1 |_{E^{cu}})|\ud\nu.$$
	And each Dirac measure $\nu$ supported on a singularity satisfies 
	$${\rm h}_{\nu}(\varphi_1 )=0<\int\log|\Det(D\varphi_1 |_{E^{cu}})|\ud\nu.$$
	Since the entropy is an affine function of the measure,
	for any invariant measure $\mu$, one has $\displaystyle{\rm h}_{\mu}(\varphi_1 )\le\int \sum\lambda^+(x)d\mu(x)=\int\log|\Det(D\varphi_1|_{E^{cu}})|\ud\mu.$
	The lemma follows.
\end{proof}

Now, we are ready to prove Theorem~\ref{Thm:unique-physical-on-attractor}.

\begin{proof}[Proof of Theorem~\ref{Thm:unique-physical-on-attractor}]
	Recall Theorem~\ref{thm.continuity-of-singular-hyperbolic-attractor} that for $C^1$ open and dense vector fields, a singular hyperbolic attractor is robustly transitive  and  Remark~\ref{Rem:singular-hyperbolic-attractor} that a singular hyperbolic attractor is always positively singular hyperbolic.
	For any open set $V\subset M$,
	let $\cU_V$ be the (open) set of vector fields $X$, such that
	for any $Y$ in a neighborhood of $X$, the open set $V$ is trapping and its maximal invariant set $\Lambda_V^Y$ is transitive
	and singular hyperbolic.
	For $X\in \cU_V$ we denote
	
	$$\cE_V(X)=\big\{\mu\in\cM_{\rm inv}(X):~\mu(V)=1 \textrm{~and~}
	{\rm h}_\mu(\varphi_1^X)\ge \int\log|\Det(D\varphi_1^X|_{E^{cu}})|\ud\mu\big\}.$$
	
	\begin{Claim}
		The set  $\cE_V(X)$ is compact in $\cM_{\rm inv}(X)$ and
		$X\mapsto \cE_V(X)$ is upper semi-continuous on $\cU_V$
		(for the Hausdorff topology).
		In particular the set $\cR_V$ of continuity points of the map $\cE_V$ is a dense G$_\delta$ subset of $\cU_V$.
	\end{Claim}
	\begin{proof}
To prove the claim, it suffices to show that
for any vector fields $X_n\in \cU_V$ with measures $\mu_n\in\cE_V(X_n)$,
if $X_n\to X$ and $\mu_n\to \mu$, then $\mu\in\cE_V(X)$.
Since each measure $\mu_n$ is supported on $\Lambda_V^{X_n}$, the measure
$\mu$ is supported on $\Lambda_V^X$.
		Moreover $\Lambda_V^{X_n}$ admits a singular hyperbolic splitting $E^{ss,X_n}\oplus E^{cu,X_n}$, which is the continuation of the singular hyperbolic splitting $E^{ss}\oplus E^{cu}$ over $\Lambda_V^X$. By the upper semi-continuity of metric entropy (Theorem~\ref{Thm;ups-entropy})
$$
			{\rm h}_{\mu}(\varphi_1^X)\ge\limsup_{n\to\infty}{\rm h}_{\mu_n}(\varphi_1^{X_n})
			\ge\lim_{n\to\infty}\int\log|\Det(D\varphi_1^{X_n}|_{E^{cu,X_n}})|{\rm d}\mu_n = \int\log|\Det(D\varphi_1^{X}|_{E^{cu}})|{\rm d}\mu.
$$
		This implies $\mu\in\cE_V(X)$.
	\end{proof}
	
	Let $\cG_0$ be the  open and dense subset of $\cX^1(M)$ provided by Theorem~\ref{thm.continuity-of-singular-hyperbolic-attractor}:
	for any $X\in \cG_0$, any singular hyperbolic attractor is robustly transitive.
	
	Let $\cV$ be a countable family of open sets of $M$
	such that for any $V_0\subset M$ open and $K_0\subset V_0$ compact, there exists $V\in \cV$
	with $K_0\subset V\subset V_0$. We introduce the dense G$_\delta$ subset:
	$$\cG:=\cG_0\cap \bigcap_{V\in\cV} \bigg(\cR_V\cup (\cX^1(M)\setminus \overline{\cU_V})\bigg)\subset \cX^1(M).$$
	
	Let us consider $X\in \cG$ with a singular hyperbolic attractor $\Lambda$. Since $X\in \cG_0$,
	the set $\Lambda$ is robustly transitive, hence there exists $V\in \cV$ such that for any $Y$ close to $X$,
	the domain $V$ is trapping and
	the maximal invariant set $\Lambda_V^Y$ in $V$ is transitive and singular hyperbolic.
	In particular $X$ belongs to $\cU_V$, and thus in $\cR_V$ (by definition of $\cG$).
	
	Now we take a $C^2$ vector field $Y$ that is close to $X$. By Theorem~\ref{Thm.pseudo-physical-measure-attractor}
	and Lemma~\ref{Lem:equivalent-description}, $\Lambda^Y_V$ supports a unique measure $\mu^Y$  such that ${\rm h}_{\mu^Y}(\varphi_1^Y)\geq\int\log|\Det(D\varphi_1^Y|_{E^{cu}})|\ud\mu^Y$.
	This gives $\cE_V(Y)=\{\mu^Y\}$.
	Since the map $\cE_V$ is continuous at $X$, one deduces that $\cE_V(X)$ is a singleton.
Together with Lemma~\ref{Lem:equivalent-description}, one deduces that 
$\Lambda=\Lambda_V^X$ supports exactly one SRB measure.
By~Theorem~\ref{Thm.pseudo-physical-measure-attractor}, Lebesgue-almost every point $x$ in the basin of $\Lambda$
satisfies \[\displaystyle\lim_{t\rightarrow+\infty}\frac{1}{t}\int_{0}^t\delta_{\varphi_s(x)}\ud s=\mu.\]
In particular $\mu$ is a physical measure and we have $\Leb(\Basin(\mu))=\Leb(\Basin(\Lambda))$.
Note that for any physical measure $\nu$ supported on $\Lambda$,
$\Basin(\nu)\subset \Basin(\Lambda)$. This implies that $\mu$ is the unique physical measure
supported on $\Lambda$.
\end{proof}

\subsection{Existence and accumulation of attractors}
We now give a new criterion for the existence of attractors, based on invariant manifolds of bi-Pliss points. 
\begin{Theorem}~\label{p.unstable-manifold-attractor}
There exists a dense G$_\delta$ subset $\cG\subset\cX^1(M)$ such that for any $X\in\cG$ with a multi-singular hyperbolic chain recurrence class $C$,
if there are vector fields $X_n\to X$ with hyperbolic periodic orbit $\gamma_n$ satisfying:
\begin{itemize}
\item[] any neighborhood $U$ of $C$ contains $W^u(\gamma_n)$ for some arbitrarily large $n$,
\end{itemize}
then $C$ is an attractor of $X$.
\end{Theorem}
\begin{proof}
	Let $\cG$ be a dense G$_\delta$ subset of $\cX^1(M)$ such that every $X\in\cG$ satisfies Theorem~\ref{thm.generic-property-of-star}, Theorem~\ref{Thm:go-to-Lyapunov} and Proposition~\ref{Pro:bi-Pliss}.
	Let $X\in\cG$ and let $C$ be a chain recurrence class of $X$ that satisfies the assumptions in Theorem~\ref{p.unstable-manifold-attractor}. 
	
	Let us first assume that the period of $\gamma_n$ is uniformly bounded from above by some $\tau>0$.
	By the finiteness and hyperbolicity of singularities and periodic orbits of period less than $\tau$, up to considering a subsequence $\gamma_n$ is the hyperbolic continuation for $X_n$ of a hyperbolic periodic orbit $\gamma$ of $X$ which is contained in $C$.
	If $\gamma$ is a sink, then $C=\gamma$ is an attractor of $X$. Otherwise $\gamma$ and  $\gamma_n$
	have non-trivial unstable manifolds; by our assumptions $\limsup W^u(\gamma_n)\subset C$, one gets
	$W^u(\gamma)\subset C$. By   Item~\ref{i.lyapunov-stable-star} of Theorem~\ref{thm.generic-property-of-star}, the chain recurrence class $C$ is an attractor.
	
	From now on, we assume that the period of $\gamma_n$ tends to infinity. Thus, $C$ is not reduced to a singularity nor to a periodic orbit.
	By Theorem~\ref{Thm:fundamental} and Corollary~\ref{Lem:bi-Pliss-exsitence},
	there are $\eta,T>0$ such that for $n$ large the hyperbolic splitting $\cN_{\gamma_n}=\cN^s\oplus \cN^u$ on the periodic orbit $\gamma_n$
	is $(\eta,T)$-dominated; moreover $\gamma_n$ contains an $(\eta,T)$-bi-Pliss point $p_n$
	for $(\Psi^{*,X_n}_t)_{t\in \RR}$.
	
Taking a subsequence, $(p_n)$ converges to a point $z$.
Several cases will be considered.
	
\paragraph{Case I:  $z$ is not a singularity of $X$.}
By Proposition~\ref{Pro:limit-bi-pliss}, $z$ is $(\eta,T)$-bi-Pliss for $(\Psi^{*,X}_t)_{t\in \RR}$.
	Since each $p_n$ is $(\eta,T)$-bi-Pliss for $X_n$,
	by Proposition~\ref{Pro:bi-Pliss},   arbitrarily close to $z$, there is an $(\eta/2,T)$-bi-Pliss periodic point $p$ of $X$. By  Proposition~\ref{Pro:size-stable-Pliss},  {there is $\delta>0$ associated to $\eta/2$ and $T$ such that there are sub-manifolds 
		$$W^{\cN^s}_{\delta\|X(p)\|}(p,X),~W^{\cN^u}_{\delta\|X(p)\|}(p,X),~W^{\cN^s}_{\delta\|X(z)\|}(z,X),~W^{\cN^u}_{\delta\|X(z)\|}(z,X)$$
		satisfying the following properties:
		\begin{itemize}
			\item $W^{\cN^s}_{\delta\|X(p)\|}(p,X), W^{\cN^s}_{\delta\|X(z)\|}(z,X)$ are contained in the stable sets of the orbit of $p$ and $z$ respectively;
			\item $W^{\cN^u}_{\delta\|X(p)\|}(p,X), W^{\cN^u}_{\delta\|X(z)\|}(z,X)$ are contained in the unstable sets of the orbit of $p$ and $z$ respectively.
		\end{itemize}
		As $p$ can be chosen arbitrarily close to $z$, by the last part of Proposition~\ref{Pro:size-stable-Pliss},
		$W^{\cN^s}_{\delta\|X(p)\|}(p,X)$ intersects $\varphi_{(-1,1)}(W^{\cN^u}_{\delta\|X(z)\|}(z,X))$ transversely and $W^{\cN^s}_{\delta\|X(z)\|}(z,X)$ intersects $\varphi_{(-1,1)}(W^{\cN^u}_{\delta\|X(p)\|}(p,X))$ transversely. This implies that $p$ and $z$ are in the same chain recurrence class.
	}Moreover,
		by Proposition~\ref{Pro:size-stable-Pliss}, $W^{\cN^u}_{\delta\|X_n(p_n)\|}(p_n,X_n)$ is contained in the unstable set of $\gamma_n$, hence

		$$W^{\cN^u}_{\delta\|X(z)\|}(z,X)\subset\limsup_{n\to\infty}W^{\cN^u}_{\delta\|X_n(p_n)\|}(p_n,X_n)\subset\limsup_{n\to\infty}\overline{W^u(\gamma_n)}\subset C.$$
	
	Since $\varphi_{(-1,1)}(W^{\cN^u}_{\delta\|X(z)\|}(z,X))$ intersects $W^s({\rm Orb}(p))$ transversely, by the $\lambda$-lemma,  the unstable manifold of $ {\rm Orb}(p)$ is accumulated by forward iterates of  $\varphi_{(-1,1)}(W^{\cN^u}_{\delta\|X(z)\|}(z,X))$. Thus $W^u(\orb(p))$ is contained in $C$. By  Item~\ref{i.lyapunov-stable-star} of Theorem~\ref{thm.generic-property-of-star}, $C$ is an attractor.
	
\paragraph{Case II:	$z=\sigma_+$ is a Lorenz like singularity with negative center Lyapunov exponent.}
By the multi-singular hyperbolicity, $\ind(\sigma_+)$ is larger than the stable dimension of $\gamma_n$.
	By Theorem~\ref{Thm:intersection-singular-periodic}, $W^s(\sigma_+^{X_n})$
	intersects $W^u(\gamma_n)$ transversely for $n$ large enough, where $\sigma_+^{X_n}$ is the continuation of $\sigma_+$ for $X_n$.
	By the $\lambda$-lemma,
	one deduces that $W^u(\sigma_+^{X_n})\subset \overline{W^u(\gamma_n)}$,
	and by our assumptions,  $\limsup \overline{W^u(\sigma_+^{X_n})}\subset C$.
	By the continuity of the unstable manifold of a hyperbolic singularity, one gets $W^u(\sigma_+)\subset C$.
	One concludes as in the previous case with  Theorem~\ref{thm.generic-property-of-star} that the chain recurrence class $C$ is an attractor.
	
\paragraph{Case III: $z=\sigma_-$ is a Lorenz like singularity with positive center Lyapunov exponent.}
The unstable space of $\sigma_-$ decomposes $E^u_{\sigma_-}=E^c_{\sigma_-}\oplus E^{uu}_{\sigma_-}$, with $\dim(E^c_{\sigma_-})=1$.
\begin{Claim}
Up to considering a subsequence, there are $\eta',T'>0$ such that the orbit of each $p_n$ contains 
some $(\eta',T',\cN^{u})$-Pliss point $q_n$ for $(\Psi^{*,X_n}_t)_{t\in \RR}$
and the sequence $(q_n)$ converges to some point $x\in W^u(\sigma_-)\setminus\{\sigma_-\}$.
\end{Claim}
\begin{proof}
By Lemma~\ref{l.limit-singularity},
the accumulation set of the unit vectors $\{X_n(p_n)/\|X_n(p_n)\|\}$ is contained
in $E^s_{\sigma_-}\oplus E^{c}_{\sigma_-}$.
Since the splitting $\cN_{\gamma_n}=\cN^s\oplus \cN^u$ is uniformly bounded,
if $V$ is a small neighborhood of $\sigma_-$, to compare with splitting in $T_{\sigma_-}M$, one has 
 for any point $z\in \gamma_n\cap V$, $\RR\cdot X(z)\oplus \cN^s(z)$ is close to $E^s_{\sigma_-}\oplus E^{c}_{\sigma_-}$, hence the angle between $\cN^u(z)$ and $E^s_{\sigma_-}\oplus E^{c}_{\sigma_-}$
is uniformly bounded away from below. Hence there are $\eta',T'>0$ such that for any
$z\in \gamma_n$ satisfying $\varphi_{[-T',0)}(x)\subset V$,
\begin{equation}\label{e.domination-sigma}
\|\Psi_{-T'}^{*,X_n}|_{\cN^u(x)} \|=\frac{\|\Psi_{-T'}^{X_n}|_{\cN^u(x)} \|}{\|D\varphi^{X_n}_{-T'}|_{X_n(x)}\|}\leq e^{-\eta'T'/T}.
\end{equation}
We can furthermore require that $T'$ is a multiple of $T$ and $\eta'<\eta$.

Let $q_n$ be the largest iterate $\varphi_{kT'}^{X_n}(p_n)$ of $p_n$
under $\varphi_{T'}^{X_n}$ such that
$\varphi_{t}^{X_n}(p_n)\in V$ for each $0\leq t\leq kT'$.
Since $p_n$ is an $(\eta',T')$-bi-Pliss point for $(\Psi^{*,X_n}_t)_{t\in \RR}$,
\eqref{e.domination-sigma} implies that
$q_n$ is an $(\eta',T',\cN^u)$-Pliss point for $(\Psi^{*,X_n}_t)_{t\in \RR}$ as well.

By construction, the points $q_n$ avoid a small neighborhood of $\sigma_-$.
As $p_n\to \sigma_-$, by taking a subsequence, the points $q_n$ converge to some point
$x$ whose negative orbit under $\varphi^{X}$ is contained in $\overline{V}$.
This implies that $x\in W^u(\sigma_-)\setminus \{\sigma_-\}$.
\end{proof}
	By Proposition~\ref{Pro:limit-bi-pliss}, $x$ is an $(\eta',T',\cN^{u})$-Pliss point for $(\Psi^{*,X_n}_t)_{t\in \RR}$. By Proposition~\ref{Pro:size-stable-Pliss}, there is $\delta'>0$ such that 
	$$W^{\cN^u}_{\delta'\|X(x)\|}(x,X)=\lim_{n\to\infty}W^{\cN^u}_{\delta'\|X_n(q_n)\|}(q_n,X_n)$$ 
	is contained in the unstable set of the orbit of $x$.
	Moreover $W^{\cN^u}_{\delta'\|X_n(q_n)\|}(q_n,X_n)$ is contained in the unstable set of $q_n$ for $X_n$.
	Hence by our assumptions we get:
	$$W^{\cN^u}_{\delta'\|X(x)\|}(x,X)\subset\limsup_{n\to\infty}W^{\cN^u}_{\delta'\|X_n(q_n)\|}(q_n,X_n)\subset C(\sigma_-)=C.$$
	
	Since $x$ is contained in $W^u(\sigma_-)$, one has $W^{\cN^u}_{\delta'\|X(x)\|}(x,X)\subset W^u(\sigma_-)$.  On the other hand, $W^{\cN^u}_{\delta'\|X(x)\|}(x,X)$ has dimension $\dim\cN^u=(\dim E^u(\sigma_-)-1)$
	and  is  transverse to the vector field $X$ inside $W^u(\sigma_-)$. It is clear that $\varphi^X_\RR(W^{\cN^u}_{\delta'\|X(x)\|}(x,X))$ forms an open subset of $W^u(\sigma_-)$. Thus, by Theorem~\ref{Thm:go-to-Lyapunov}, there is a point $y\in \varphi_\RR(W^{\cN^u}_{\delta'\|X(x)\|}(x,X))$ such that $\omega(y)$ is contained in a Lyapunov stable chain recurrence class. This implies that $C$ is Lyapunov stable and hence is an attractor by Item~\ref{i.attractor-star} of Theorem~\ref{thm.generic-property-of-star}.
	However, by Remark~\ref{Rem:singular-hyperbolic-attractor}, all singularities in $C$ are  Lorenz like with negative center Lyapunov exponent which contradicts with $\sigma^-\in C$, and thus Case III does not happen.
\end{proof}

\subsection{Finiteness of attractors for generic multi-singular hyperbolic vector fields}
In~\cite{PYY}, the authors obtain the finiteness of attractors for $C^1$ generic star vector fields. In this section, we provide another proof different from that in~\cite{PYY}. By Combining Theorems~\ref{thm.continuity-of-singular-hyperbolic-attractor} and~\ref{p.unstable-manifold-attractor}, we can improve the statement a little bit.

\begin{Theorem}\label{Thm:main-topology}
There exists a $C^1$ open and dense subset $\cU$ of $\cX^*(M)$ such that
any $X\in\cU$ admits only finitely many quasi-attractors.
Each of them is a robustly transitive attractor and is either a singular hyperbolic homoclinic class or a sink.
The number of attractors is locally constant 
and the attractors vary continuously on $\cU$ for the Hausdorff topology.
\end{Theorem}
\begin{proof}
The number of quasi-attractors is an element of
$\NN\cup\{\infty\}$ which vary semi-continuously with the vector field $X$.
So for each integer $k$, the set $\cV_k\subset \cX^*(M)$ of vector fields with at least $k$
quasi-attractors is open. We then define inductively for $k\in \NN$
$$\cU_1=\cX^*(M)\setminus \overline{\cV_2}=\cV_1\setminus \overline{\cV_2},$$
$$\cU_k=\cV_k\setminus (\overline{\cV_{k+1}}\cup \overline{\cU_1}\cup\dots\cup \overline{\cU_{k-1}}).$$
Each set $\cU_k$ is open and for any $Y\in \cU_k$
there exists exactly $k$ quasi-attractors.

\begin{Claim}
$\bigcup_{k\in\NN} \cU_k$ is dense in $\cX^*(M)$.
\end{Claim}
\begin{proof}
Let us assume by contradiction that the open set
$$\cU_\infty=\cX^*(M)\setminus \overline{\bigcup \cU_k}$$
is non empty.
Since $\cX^*(M)=\overline{(\cU_1\cup\dots\cU_k)\cup \cV_{k+1}}$,
the set $\cG_0:=\cU_\infty\setminus (\cup_k \overline{\cV_k}\setminus \cV_k)$
is contained in $\cap _k \cV_k$.
Let $\cG_1$ be the dense G$_\delta$ subset of $\cX^*(M)$
which satisfies the properties in Theorem~\ref{p.unstable-manifold-attractor}.
The set $\cG=\cG_0\cap \cG_1$ is a dense G$_\delta$ subset of $\cU_\infty$.

Let us consider $X\in \cG$. By Theorem~\ref{thm.finite-sink}, $X$ admits at most finitely many sinks and finitely many quasi-attractors with singularities. Hence $X$ has infinitely many quasi-attractors $\{\Lambda_n\}_{n\in\NN}$ without singularities and which are not sinks. By Propositions~\ref{Prop:multi-singular-hyperbolic-no-singularity} and~\ref{Prop:hyperbolic}, each $\Lambda_n$
is a hyperbolic homoclinic class $H(\gamma_n)$ of a hyperbolic periodic orbit $\gamma_n$ of saddle type and moreover $W^u(\gamma_n)\subset\Lambda_n$.	
Up to taking a subsequence, one will assume that the sequence $(\Lambda_n)$ accumulates on a subset of a chain recurrence class $\Lambda$. Applying Theorem~\ref{p.unstable-manifold-attractor} to $X_n=X$ and $\gamma_n$, one deduces that $\Lambda$ is an attractor, which contradicts to the isolation of attractors. Therefore
$\cU_\infty$ is empty.
\end{proof}

Let $\cU'\subset \cX^1(M)$ be the subset
of vector fields provided by Theorem~\ref{thm.continuity-of-singular-hyperbolic-attractor}
and let us define the open and dense subset of $\cX^*(M)$:
$$\cU=\cU'\cap \big(\bigcup \cU_k\big).$$
Take $X\in \cU$. By the previous construction, $X$ has finitely many quasi-attractors and their
number is locally constant. Each quasi-attractor $\Lambda$ which is not a sink is singular hyperbolic
by Lemma~\ref{Lem:quasi-attractor-singular-hyperbolic}.
Then Theorem~\ref{thm.continuity-of-singular-hyperbolic-attractor} implies that $\Lambda$
is a robustly transitive attractor which is a homoclinic class.
In particular $\Lambda$ is the maximal invariant set in a compact neighborhood $U$.
Then the map $Y\mapsto \cap_t \varphi_t^Y(U)$ associates to any $Y$ close to $X$
an attractor $\Lambda^Y$ which is the continuation of $\Lambda$.
By definition this map is upper semi-continuous.
Since each $\Lambda^Y$ is a homoclinic class, this map is also
lower semi-continuous. Hence the attractors $\Lambda_i^Y$ vary continuously with $Y$.
\end{proof}

\subsection{Proof of Theorem~\ref{Thm:main-restated}}
Now, we are ready to prove the Theorem~\ref{Thm:main-restated}
(which implies Theorems~\ref{Thm:main-ergodic} and~\ref{Thm:main-generic-ergodic}).

Let $\cU_0\subset \cX^*(M)$ be the $C^1$ open and dense subset given by
Theorem~\ref{Thm:main-topology}.
Any vector field $X\in \cU_0$ has finitely many attractors $\Lambda_1,\dots,\Lambda_k$,
and each of them is a sink or a singular hyperbolic homoclinic class.

\begin{Claim}
For any $C^2$ vector field $X\in \cU_0$,
each attractor $\Lambda_i$ which is not a sink supports exactly one SRB measure $\mu_i$.
This measure is physical, its support coincides with $\Lambda_i$
and
\begin{equation}\label{e.basin}
\Leb(\Basin(\mu_i))=\Leb(\Basin(\Lambda_i^Y)).
\end{equation}
\end{Claim}
\begin{proof}
By Theorem~\ref{Thm.pseudo-physical-measure-attractor}, $\Lambda_i$ admits a unique physical measure $\mu_i$.
The measure satisfies~\eqref{e.basin} and~\eqref{e.entropy-formula}, hence is a hyperbolic SRB
by Proposition~\ref{Prop:sum-positive-jacobian} and Lemma~\ref{Lem:equivalent-description}.
By Corollary~\ref{Cor:SRB-ergodic-class}, there exists a hyperbolic periodic orbit $\gamma$ such that
$H(\gamma)\subset \overline{W^u(\gamma)}\subset \supp(\mu_i)\subset \Lambda_i$. By Theorem~\ref{thm.continuity-of-singular-hyperbolic-attractor},
$H(\gamma)=\Lambda_i^Y$, so $\supp(\mu_i)	= \Lambda_i$.
\end{proof}

\begin{Claim}
There exists a $C^1$ open and dense subset $\cU\subset \cU_0$ such that
for any $C^\infty$ vector field $X\in \cU$,
the attractors $\Lambda_1,\dots,\Lambda_k$ satisfy:
$${\rm Leb}\big(\bigcup_{1\le i\le k}{\rm Basin}(\Lambda_i)\big)={\rm Leb}(M).$$\end{Claim}
\begin{proof}
Let $\cG$ be the dense G$_\delta$ subset of $\cU_0$ of vector fields satisfying
Theorem~\ref{p.unstable-manifold-attractor} and let us assume by contradiction that
the claim fails. Then there exists a sequence $Y_n$ of $C^\infty$ vector fields which converge
in the $C^1$ topology to some $X\in \cG$ such that 
		$${\rm Leb}\big(\bigcup_{1\le i\le k}{\rm Basin}(\Lambda^{Y_n}_i)\big)<{\rm Leb}(M).$$
The attractors $\Lambda^{Y_n}_i$ are the continuations of the attractors
$\Lambda^{X}_1,\dots,\Lambda^{X}_k$ for $X$ and admit trapping neighborhood
$U_1,\dots,U_k$ that do not depend on $Y_n$ for $Y_n$ close to $X$.

		Each vector field $Y_n$ is multi-singular hyperbolic. By Proposition~\ref{Prop:contable-class}, each $Y_n$ has at most countably many chain recurrence classes.
		Moreover the $\omega$-limit set of any point is contained in a chain recurrence class.
		Hence there is a chain recurrence class $C_n\subset M\setminus\bigcup\limits_{1\le i\le k}{\rm Basin}(\Lambda_i^{Y_n})$ of $Y_n$ such that ${\rm Leb}(\Basin(C_n))>0$ and $C_n$ is not an attractor of $Y_n$.
		By Proposition~\ref{Prop:lebesgue-basin}, $C_n$ cannot be hyperbolic. Consequently, $C_n$ contains a singularity $\sigma_n$ of $Y_n$. Since all singularities of $X$ are hyperbolic, $\sigma_n$ is the continuation of a singularity of $X$. By taking a subsequence if necessary, one can assume that $C_n$ is the continuation of the
		chain recurrence class $C(\sigma)$ of some $\sigma\in{\rm Sing}(X)$.

By Theorem~\ref{Thm:main-localized} and Corollary~\ref{Cor:SRB-ergodic-class}, since $Y_n$ is $C^\infty$, there exists a periodic orbit $\gamma_n\subset C_n$ for $Y_n$ with $W^u(\gamma_n)\subset C_n$. Taking a subsequence if necessary, one can assume that the sequence $(C_n)$ accumulates on a subset of a chain recurrence class of $C$ of $X$. By the choice of $C_n$, the set $C$ is disjoint from
the union of neighborhoods $\bigcup_{1\le i\le k}U_i$, hence $C$ is disjoint from the attractors
$\Lambda^X_1,\dots,\Lambda^X_k$. By Theorem~\ref{p.unstable-manifold-attractor}, the class $C$ is an attractor which is a contradiction since $\Lambda_1^X,\dots,\Lambda_k^X$ are the only attractors of $X$.
\end{proof}

The two previous claims conclude the proof of Theorem~\ref{Thm:main-restated} for $C^\infty$ vector fields in $\cU$.
\bigskip

	Let us consider for any $X\in \cX^1(M)$, the Lebesgue measure $L(X)$
	of the union of the basins of all the robustly transitive attractors of $X$.
	Since the robustly transitive attractors persist for vector fields $Y$ that are $C^1$ close to $X$
	and since any compact set contained in the basin of a robustly transitive attractor for $X$
	is still contained in the basin of
	an attractor for $Y$, the map $L\colon \cX^1(M)\to \RR$ is lower semi-continuous.
	Hence there is a dense  G$_\delta$ subset $\cG_0\subset \cX^1(M)$ of continuity points of $L$.

	Let $\cG_1\subset \cG_0\cap \cU$ be a dense G$_\delta$ subset of $\cX^*(M)$ such that every vector field in $\cG_1$ satisfies Theorem~\ref{Thm:unique-physical-on-attractor}.
Let us fix such a $X\in \cG_1$.
By the previous claim, for any $C^\infty$ vector field $Y$ that is $C^1$ close to $X$,
	the continuations of the attractors of $X$ satisfy
	${\rm Leb}(\bigcup_{i=1}^k{\rm Basin}(\Lambda_i^Y))={\rm Leb}(M)$.
	In particular $L(Y)={\rm Leb}(M)$. Since $X$ is a continuity point of the map $L$,
	one deduces $L(X)={\rm Leb}(M)$, which gives
	$${\rm Leb}(\bigcup_{i=1}^k{\rm Basin}(\Lambda^X_i))={\rm Leb}(M).$$

Since $X$ satisfies Theorem~\ref{Thm:unique-physical-on-attractor},
any singular hyperbolic attractor $\Lambda^X_i$ supports exactly one SRB measure $\mu_i$; moreover $\mu_i$ is physical and ${\rm Leb}({\rm Basin}(\mu_i))={\rm Leb}({\rm Basin}(\Lambda_i))$.
The properties 2 and 3 of Theorem~\ref{Thm:main-restated} are thus satisfied by any $X\in \cG_1$.
It remains to describe the support of the SRB measures $\mu_i$.
\medskip

Fix a countable basis $\{O_\ell\}_{\ell\in\mathbb{N}}$ of open subset of $M$ and consider  the countable family $\{V_n\}_{n\in\mathbb{N}}$  which consists of all the possible unions of finite elements in $\{O_\ell\}_{\ell\in\mathbb{N}}$.

By the robustness and continuity of the attractors, there exist
countably many open and disjoint subsets $\cU(j)\subset \cU$
such that  $\bigcup_j \cU(j)$ is (open and) dense in $\cU$
with the following properties:
for each of $\cU(j)$ there exist $k\geq 1$ and continuous maps
$Y\mapsto \Lambda_{i}^Y$, with $1\leq i \leq k$, associating the attractors of the vector fields in $\cU(j)$.

Let $\cP_j$ be the set of $Y\in \cU_j$
such that for each $1\leq i\leq k$
there exists a unique SRB measure $\mu_i^Y$ supported on $\Lambda_i^Y$.
For each open set $V_n$ and each $i$ we define
$$\cP_{j,i,n}=\{Y\in \cP_j, \mu_i^Y(V_n)>0\}.$$
By the claim in the proof of Theorem~\ref{Thm:unique-physical-on-attractor}, this is an open set of $\cP_j$.
We know that $\cP_j$ contains a dense G$_\delta$ subset of $\cU(j)$.
Hence the following set is residual in $\cU\subset \cX^*(M)$:
$$\cG= \cG_1\cap \bigcup_j \bigcap_{i,n} \big(\cP_{j,i,n}\cup (\cU(j)\setminus \overline{\cP_{j,i,n}})\big).$$

Let us fix $X\in \cG$ (contained in some $\cU(j)$) and consider an SRB measure $\mu_i$ which is supported
on an attractor $\Lambda_i^X$ which is not a sink.
Let $V_n$ be an open set which intersects $\Lambda^X_i$.
Let us consider a sequence of $C^\infty$ vector fields $Y_m$ which converge to $X$ in the $C^1$ topology.
By the first claim of this section, $Y_m$ belongs to $\cP_j$.
By continuity of the attractors, for $m$ large the set $\Lambda_i^{Y_m}$ intersects $V_n$.
The SRB measure for $Y_m$ supported on $\Lambda_i^{Y_m}$ has full support,
hence gives positive mass to $V_n$. So the vector fields $Y_m$ for $m$ large belong to $\cP_{j,i,n}$.
Since $X\in \cG$, one deduces $X\in \cP_{j,i,n}$, hence $\mu_i(V_n)>0$.
This proves that the support of $\mu_i$ coincides with $\Lambda_i^X$,
giving property 1 for any $X\in \cG$.

The proof of Theorem~\ref{Thm:main-restated} is complete.
\qed

\appendix

\section{Limit sets of generic points}
In this section we prove Theorem~\ref{Thm:go-to-Lyapunov}.

We first state the connecting lemma for vector fields from  ~\cite{WX,W02}.
\begin{Theorem}\label{Lem:connecting}
Given $X\in\cX^1(M)$ and any $C^1$ neighborhood $\cU$ of $X$, there exist $T>0$, $\rho\in(0,1)$
and for any point $x\in M$ which is not critical there exists $\delta_0>0$
such that for any $\delta\leq\delta_0$, the orbit-tube  $$\Delta_{T,\delta}(x):=\bigcup\limits_{t\in[0,T]}\varphi^X_t(B_\delta(x))$$ satisfies  the following property. For any two points $p,q\notin\Delta_{T,\delta}$, if both the forward $\varphi^X$-orbit of $p$ and the backward $\varphi^X$-orbit of $q$ intersect
	$B_{\rho\cdot\delta}(x)$, then there exists $Y\in\cU$ such that $q$ belongs to the forward $\varphi^Y$-orbit of $p$. Moreover, $Y(y)=X(y)$ for $y\in M\setminus\Delta_{T,\delta}(x)$.
\end{Theorem}

The following is the flow version of the connecting lemma for pseudo orbits in~\cite{BC}.
We recall that the notation $x\dashv y$ has been defined in Section~\ref{ss.chain-recurrence}.
\begin{Theorem}[Th\'eor\`eme 1.1 in \cite{BC}]\label{Thm:connecting-pseudo}
	There exists a dense G$_\delta$ subset $\cG_0$ of $\cX^1(M)$ such that for any $X\in\cG_0$ and any two points $x,y\in M$, if $x\dashv y$ under the flow $(\varphi^X_t)_{t\in\RR}$, then for any neighborhood $U$ of $x$ and any neighborhood $V$ of $y$, there exist $z\in U$ and $t_0>0$ such that $\varphi^X_{t_0}(z)\in V$.
\end{Theorem}

Now we manage to prove Theorem~\ref{Thm:go-to-Lyapunov}.

\begin{proof}[Proof of Theorem~\ref{Thm:go-to-Lyapunov}]
	
	Fix a countable basis $\{O_n\}_{n\in\mathbb{N}}$ of open subset of $M$ and consider  the countable family $\{U_n\}_{n\in\mathbb{N}}$  which consists of all the possible   unions of finite elements in $\{O_n\}_{n\in\mathbb{N}}$. For each $k,m,n\in\mathbb{N}$, one defines the following open and disjoint subsets of $\cX^1(M)$:
	\begin{itemize}
		
		\item $\cI_{k,m,n}$: the vector field $Y$ belongs to $\cI_{k,m,n}$ if and only if  $Y$ has a hyperbolic critical element $\gamma_Y$ in $U_k$  such that there exists $x\in W^u(\gamma_Y)\cap U_m$ satisfying $\orb^+(x,\varphi^Y)\cap U_n\neq\emptyset$, where $\orb^+(x,\varphi^Y)$ denotes the positive orbit of $x$ under the flow $(\varphi^Y_t)_{t\in\RR}$.
		
		\item $\cN_{k,m,n}$: the vector field $Y$ belongs to $\cN_{k,m,n}$ if and only if  there exists a $C^1$ neighborhood $\cV_Y$ of $Y$ such that for any $Z\in\cV_Y$, \\
		-- either $Z$ has no hyperbolic critical elements in $U_k$, \\
		-- or for each  hyperbolic  critical element  $\gamma_Z$ of $Z$ in $U_k$, any point $x$ in the set $ W^u(\gamma_Y)\cap U_m$ (maybe empty) satisfies $\orb^+(x,\varphi^Y)\cap U_n=\emptyset$.
		
	\end{itemize}
Note that $\cN_{k,m,n}$ is the complement of $\overline{\cI_{k,m,n}}$,
hence $\cI_{k,m,n}\cup\cN_{k,m,n}$ is an open and dense subset of  $\cX^1(M)$.
Let  $\cG_0$   be the dense G$_\delta$ set such that any $X\in\cG_0$ satisfies Theorem~\ref{Thm:connecting-pseudo}  and let $\cG_{KS}$ be the set of Kupka-Smale vector fields.
We define the following dense G$_\delta$ subset of $\cX^1(M)$:
	$$\cG=\cG_0\cap\cG_{KS}\cap\big(\bigcap_{k,m,n\in\NN}(\cI_{k,m,n}\cup\cN_{k,m,n})\big).$$
We show that every $Y\in \cG$ satisfies the announced properties in Theorem~\ref{Thm:go-to-Lyapunov}.
	
	Let  $Y\in \cG$ and  $\gamma$ be a hyperbolic critical element of $Y$. There exist  a neighborhood $U_\gamma$ of $\gamma$ and a neighborhood $\cU_Y$ of $Y$ such that any $Z\in\cU_Y$ has only one critical element $\gamma_Z$ in $U_\gamma$: it is the maximal invariant set in $U_\gamma$ and the hyperbolic continuation of $\gamma$. Up to reducing $\cU_{Y}$, one can take $k\in\NN$ such that $U_k\subset U_\gamma$ is a neighborhood of  $\gamma_Z$  for all $Z\in\cU_{Y}$.
	For any integer $n\in\mathbb{N}$, we will define the following sets
	$$W_{n}(\gamma)=\big\{x\in W^u(\gamma):\orb^+(x,\varphi^Y)\cap U_n\neq\emptyset\big\} \textrm{~and~}R_\gamma^{n}=W_{n}(\gamma)\cup \big(W^u(\gamma)\setminus \overline{W_{n}(\gamma)}\big).$$
By definition, the set $R_\gamma^{n}$ is an open and dense subset of $W^u(\gamma)$,
and	therefore the set  $$R_\gamma=\bigcap_{n\in\NN}R_\gamma^{n}$$
	is a dense G$_\delta$ subset of $W^u(\gamma)$. To prove Theorem~\ref{Thm:go-to-Lyapunov}, it is sufficient to prove that $\omega(x)$ is a quasi-attractor for any $x\in R_\gamma$.

	\begin{Claim}
		For any $x\in R_\gamma$, any $y\in M$ satisfying $x\dashv y$ and any neighborhood $V$ of $y$, we have $\orb^+(x,\varphi^Y)\cap V\neq\emptyset$.
	\end{Claim}
	\begin{proof}
		By the choice of the countable family $\{U_n\}_{n\in\NN}$, there exists a sequence of nested neighborhoods $\{U_{m_i}\}_{i\in\mathbb{N}}$ of $x$ such that
		$$\bigcap_i U_{m_i}=\{x\}.$$
		Also, there exists $U_n$ such that $y\in U_n\subset V$. 
	
We first show that for any $i\in\mathbb{N}$, $Y\in\cI_{k,m_i,n}$.
Since $Y\in \cG_0$, by Theorem~\ref{Thm:connecting-pseudo}, for any neighborhood $V'$ of $x$, there exists $z\in V'$ such that $\orb^+(z,\varphi^Y)\cap U_n\neq\emptyset$.
By a standard argument, Theorem~\ref{Lem:connecting} gives,
		for any neighborhood $\mathcal{U}\subset\mathcal{U}_{\gamma}$ of $Y$, some vector field $Z\in \mathcal{U}$, and a point $x_Z\in W^u(\gamma_Z)\cap U_{m_i}$ satisfying $\orb^+(x_Z,\varphi^Z)\cap U_n\neq\emptyset$. 
		This implies that $Y$ is accumulated  by vector fields in $\cI_{k,m_i,n}$. As $Y\in \cI_{k,m_i,n}\cup \cN_{k,m_i,n}$, we have $Y\in\cI_{k,m_i,n}$.
		
Hence $x$ is accumulated by points in $W_n(\gamma)$ since $\bigcap_i U_{m_i}=\{x\}$. By the fact that $x\in R_{\gamma}^n=W_n(\gamma) \cup (W^u(\gamma)\setminus \overline{W_n(\gamma)})$, we have that $x\in W_n(\gamma)$. This implies $\orb^+(x,\varphi^Y)\cap U_n\neq\emptyset$ and thus $\orb^+(x,\varphi^Y)\cap V\neq\emptyset$.
	\end{proof}

Conley's theory shows that for any $x\in M$, there exists a quasi-attractor $\Lambda$ of $Y$ such that $x\dashv y$
for any $y\in \Lambda$ (note that for $C^1$ generic vector fields it can be deduced from Theorem~\ref{Theo:generic}.)
It remains to show that if $x\in R_\gamma$ then $\omega(x)=\Lambda$. 

By the claim above $\Lambda\subset \omega(x)$. But since $\Lambda$ is a quasi-attractor, it admits arbitrarily small
trapping neighborhoods, hence $\omega(x)=\Lambda$.
\end{proof}

\vskip 5pt

\flushleft{\bf Sylvain Crovisier} \\
\small Laboratoire de Math\'ematiques d'Orsay,  CNRS - Universit\'e Paris-Saclay, Orsay 91405, France\\
\textit{E-mail:} \texttt{Sylvain.Crovisier@universite-paris-saclay.fr}\\

\flushleft{\bf Xiaodong Wang} \\
\small School of Mathematical Sciences,  CMA-Shanghai, Shanghai Jiao Tong University, Shanghai, 200240, P.R. China\\
\textit{E-mail:} \texttt{xdwang1987@sjtu.edu.cn}\\

\flushleft{\bf Dawei Yang} \\
\small School of Mathematical Sciences,  Soochow University, Suzhou, 215006, P.R. China\\
\textit{E-mail:} \texttt{yangdw1981@gmail.com,yangdw@suda.edu.cn}\\

\flushleft{\bf Jinhua Zhang} \\
\small School of Mathematical Sciences,  Beihang University, Beijing 100191, P.R. China\\
\textit{E-mail:} \texttt{Jinhua$\_$zhang@buaa.edu.cn}\\

\end{document}